\documentclass[12pt]{amsart}
\usepackage{amsmath,amssymb,amscd}

\usepackage{slashed}
\newcommand{\sD}{\slashed{D}}

\newcommand{\1}{\mathbb{I}}

\DeclareMathOperator{\End}{End}
\newcommand{\mJ}{\mathcal{J}}
\DeclareMathOperator{\Dom}{Dom}

\newcommand{\Cliff}{{\mathrm{Cliff}}}
\newcommand{\CCliff}{{\mathbb{C}\mathrm{liff}}}
\DeclareMathOperator{\dom}{dom}

\DeclareMathOperator{\diag}{diag}

\begin{document}
\theoremstyle{plain}
\newtheorem{thm}{Theorem}[section]
\newtheorem*{thm*}{Theorem}
\newtheorem{prop}[thm]{Proposition}
\newtheorem*{prop*}{Proposition}
\newtheorem{lemma}[thm]{Lemma}
\newtheorem{cor}[thm]{Corollary}
\newtheorem*{conj*}{Conjecture}
\newtheorem*{cor*}{Corollary}
\newtheorem{defn}[thm]{Definition}
\theoremstyle{definition}
\newtheorem*{defn*}{Definition}
\newtheorem{rems}[thm]{Remarks}
\newtheorem*{rems*}{Remarks}
\newtheorem{rem}[thm]{Remark}
\newtheorem*{rem*}{Remark}
\newtheorem*{proof*}{Proof}
\newtheorem*{not*}{Notation}
\newcommand{\npartial}{\slash\!\!\!\partial}
\newcommand{\Heis}{\operatorname{Heis}}
\newcommand{\Solv}{\operatorname{Solv}}
\newcommand{\Spin}{\operatorname{Spin}}
\newcommand{\SO}{\operatorname{SO}}
\newcommand{\ind}{\operatorname{ind}}
\newcommand{\Index}{\operatorname{index}}
\newcommand{\ch}{\operatorname{ch}}
\newcommand{\rank}{\operatorname{rank}}
\newcommand{\abs}[1]{\lvert#1\rvert}
 \newcommand{\A}{{\mathcal A}}
        \newcommand{\D}{{\mathcal D}}\newcommand{\HH}{{\mathcal H}}
        \newcommand{\LL}{{\mathcal L}}
        \newcommand{\B}{{\mathcal B}}
        \newcommand{\K}{{\mathcal K}}
\newcommand{\oo}{{\mathcal O}}
         \newcommand{\PP}{{\mathcal P}}
        \newcommand{\s}{\sigma}
        \newcommand{\coker}{{\mbox coker}}
        \newcommand{\p}{\partial}
        \newcommand{\dd}{|\D|}
        \newcommand{\n}{\Vert}
\newcommand{\bma}{\left(\begin{array}{cc}}
\newcommand{\ema}{\end{array}\right)}
\newcommand{\bca}{\left(\begin{array}{c}}
\newcommand{\eca}{\end{array}\right)}
\newcommand{\sr}{\stackrel}
\newcommand{\da}{\downarrow}
\newcommand{\tD}{\tilde{\D}}
        \newcommand{\R}{\mathbb R}
        \newcommand{\C}{\mathbb C}
        \newcommand{\h}{\mathbb H}
\newcommand{\Q}{\mathcal Q}
\newcommand{\Z}{\mathbb Z}
\newcommand{\N}{\mathbb N}
\newcommand{\tto}{\longrightarrow}
\newcommand{\ben}{\begin{displaymath}}
        \newcommand{\een}{\end{displaymath}}
\newcommand{\be}{\begin{equation}}
\newcommand{\ee}{\end{equation}}
        \newcommand{\bean}{\begin{eqnarray*}}
        \newcommand{\eean}{\end{eqnarray*}}
\newcommand{\nno}{\nonumber\\}
\newcommand{\bea}{\begin{eqnarray}}
        \newcommand{\eea}{\end{eqnarray}}
        \newcommand{\CDA}{{\mathcal C}_{{\mathcal D}}({\mathcal A})}
         \newcommand{\CDAc}{{\mathcal C}_{{\mathcal D}}({\mathcal A}_c)}
\newcommand{\supp}[1]{\operatorname{#1}}
\newcommand{\norm}[1]{\parallel\, #1\, \parallel}
\newcommand{\ip}[2]{\langle #1,#2\rangle}
\newcommand{\nc}{\newcommand}
\nc{\nt}{\newtheorem}
\nc{\gf}[2]{\genfrac{}{}{0pt}{}{#1}{#2}}
\nc{\mb}[1]{{\mbox{$ #1 $}}}
\nc{\real}{{\mathbb R}}
\nc{\comp}{{\mathbb C}}
\nc{\ints}{{\mathbb Z}}
\nc{\Ltoo}{\mb{L^2({\mathbf H})}}
\nc{\rtoo}{\mb{{\mathbf R}^2}}
\nc{\slr}{{\mathbf {SL}}(2,\real)}
\nc{\slz}{{\mathbf {SL}}(2,\ints)}
\nc{\su}{{\mathbf {SU}}(1,1)}
\nc{\so}{{\mathbf {SO}}}
\nc{\hyp}{{\mathbb H}}
\nc{\disc}{{\mathbf D}}
\nc{\torus}{{\mathbb T}}
\newcommand{\tk}{\widetilde{K}}
\newcommand{\boe}{{\bf e}}\newcommand{\bt}{{\bf t}}
\newcommand{\vth}{\vartheta}
\newcommand{\CGh}{\widetilde{\CG}}
\newcommand{\db}{\overline{\partial}}
\newcommand{\tE}{\widetilde{E}}
\newcommand{\tr}{\mbox{tr}}
\newcommand{\ta}{\widetilde{\alpha}}
\newcommand{\tb}{\widetilde{\beta}}
\newcommand{\txi}{\widetilde{\xi}}
\newcommand{\hV}{\hat{V}}
\newcommand{\IC}{\mathbf{C}}
\newcommand{\IZ}{\mathbf{Z}}
\newcommand{\IP}{\mathbf{P}}
\newcommand{\IR}{\mathbf{R}}
\newcommand{\IH}{\mathbf{H}}
\newcommand{\IG}{\mathbf{G}}
\newcommand{\CC}{{\mathcal C}}
\newcommand{\CS}{{\mathcal S}}
\newcommand{\CG}{{\mathcal G}}
\newcommand{\CL}{{\mathcal L}}
\newcommand{\CO}{{\mathcal O}}
\nc{\ca}{{\mathcal A}}
\nc{\cag}{{{\mathcal A}^\Gamma}}
\nc{\cg}{{\mathcal G}}
\nc{\chh}{{\mathcal H}}
\nc{\ck}{{\mathcal B}}
\nc{\cl}{{\mathcal L}}
\nc{\cm}{{\mathcal M}}
\nc{\cn}{{\mathcal N}}
\nc{\NN}{{\mathcal N}}
\nc{\cs}{{\mathcal S}}
\nc{\cz}{{\mathcal Z}}
\nc{\br}{{\mathcal R}}
\nc{\vf}{\varphi}
\nc{\sind}{\sigma{\rm -ind}}
\newcommand{\la}{\langle}
\newcommand{\ra}{\rangle}
\newcommand{\al}{\alpha}

\newcommand{\pert}{\mathcal{R}_\D}
\newcommand{\pertsD}{\mathcal{R}_\sD}

 \title{Pseudo-Riemannian spectral triples and the harmonic oscillator}
 
\author{Koen van den Dungen\dag, Mario Paschke, Adam Rennie\ddag} 
\thanks{\texttt{koen.vandendungen@anu.edu.au, mpaschke@spassundwissenschaft.de, renniea@uow.edu.au}}
\maketitle

\vspace{-12pt}

\centerline{\dag Mathematical Sciences Institute, Australian National University,}

\centerline{ Canberra, Australia\vspace{6pt}}

\centerline{\ddag School of Mathematics and Applied Statistics, University of Wollongong,}

\centerline{ Wollongong, Australia}

\vspace{12pt}

\centerline{{\bf Abstract}}
We define pseudo-Riemannian spectral triples,  an
analytic context broad enough to encompass a spectral description
of a wide class of pseudo-Riemannian manifolds,
as well as their noncommutative generalisations. Our main theorem shows that to each
pseudo-Riemannian spectral triple we can associate
a genuine spectral triple, and so a
$K$-homology class. With some additional assumptions we can then apply
the local index theorem. We give a range of examples and some applications. 
The example of the harmonic
oscillator in particular shows that our main theorem applies to 
much more than just classical pseudo-Riemannian manifolds.

\tableofcontents

%

\section{Introduction}
Spectral triples provide a way to extend Riemannian geometry to noncommutative
spaces, retaining the connection to the underlying topology via $K$-homology. 
In this paper we provide a definition
of pseudo-Riemannian spectral triple, enabling a 
noncommutative analogue of pseudo-Riemannian geometry,  
and show that we can still maintain contact with the underlying topology. We do this by 
`Wick rotating' to a spectral triple analogue of our pseudo-Riemannian spectral triple. 
Below we discuss the class of pseudo-Riemannian manifolds which give
examples of our construction: a key point is that we work with Hilbert spaces
not Krein spaces, and so the construction relies on a global splitting of
the tangent bundle into timelike and spacelike sub-bundles.

Our main theorem, Theorem \ref{convert}, states that one can associate a spectral triple
to a pseudo-Riemannian spectral triple via Wick rotation.
Under additional assumptions, the process of Wick rotating is shown to preserve spectral dimension, smoothness and integrability,
as we define them. Thus one obtains a $K$-homology class and the tools to compute index
pairings using the local index formula.
Since the most important Lorentzian manifolds are noncompact, we have taken care to
ensure that our definitions are consistent with the nonunital version of the local index formula, as
proved in \cite{CGRS2}. 

Section \ref{prelim} recalls what we need from the theory of nonunital spectral
triples, while Section \ref{sec:classical} recalls some pseudo-Riemannian geometry, in order to set notation
and provide motivation. Here we also show how certain pseudo-Riemannian spin manifolds provide 
examples for our theory. 

In Section \ref{sRst} we discuss some technicalities
about unbounded operators before presenting our definition of pseudo-Riemannian spectral triples.
We also provide definitions of smoothness and summability,
and a range of examples. An unexpected example
is provided by the harmonic oscillator.

Section  \ref{sRst-st} begins with our main theorem, which shows that we can obtain a spectral
triple from a pseudo-Riemannian spectral triple. We also give a sufficient condition on 
a smoothly summable 
pseudo-Riemannian spectral triple ensuring that the resulting spectral triple is
smoothly summable, so that we can employ the local index formula. 
This sufficient condition is enough for all our examples, except the harmonic oscillator.
The remainder of the section looks at the examples, in particular the oscillator, as well as
one simple non-existence result for
certain kinds of harmonic one forms on compact manifolds.

%
%

The classical examples we have presented all arise by taking a pseudo-Riemannian spin manifold
and generating a 
Riemannian metric: for this to work, we require a global splitting of the tangent
bundle into timelike and spacelike sub-bundles, and that the 
resulting Riemannian metric be 
complete and of bounded geometry.
This can always be achieved in the globally
hyperbolic case when $M_n=\R\times M_{n-1}$ provided that the induced metric on
$M_{n-1}$ is complete and of bounded geometry.

It would be desirable to have a method of producing a spectral triple, and so $K$-homology class, 
associated to a more general pseudo-Riemannian metric, or at least Lorentzian metric. 
The reason we can not 
is that, in the Lorentzian case, the weakest physically reasonable causality condition is stable
causality, and the Lorentzian metrics of such manifolds need not have complete Riemannian
manifolds associated to them by our Wick rotation procedure.

Hence to deal with general Lorentzian manifolds, 
one would have to be able to deal with Riemannian manifolds with boundary, and
even more general objects. 
This requires careful consideration of appropriate boundary conditions, and we refer to \cite{BDT} for
a comprehensive discussion of boundary conditions in $K$-homology and \cite{IL} for 
a definition of `spectral triple with boundary'. We will return to this question in a future work.

As a final comment, we remark that the purpose of this paper is to develop a framework for accessing the topological data associated to a pseudo-Riemannian manifold, commutative or noncommutative. The
details of the geometry and/or physics of such a manifold should be accessed using the pseudo-Riemannian
structure and not the associated Riemannian one.

\medskip

{\bf Acknowledgements}. We would like to thank Christian B\"{a}r, Iain Forsyth, Victor Gayral,
Rainer Verch and Ben Whale for their assistance. We also thank the referee whose comments
have led to improvements in this work. The third author also thanks the
Max Planck Institute, Leipzig, for hospitality during the initial stages of this work. The 
third author acknowledges
the support of the Australian Research Council.

\section{Nonunital spectral triples}\label{prelim}
In this section we will summarise the definitions and results
concerning nonunital spectral triples. Much of this material is from \cite{CGRS2}
where a more detailed account can be found. 

\begin{defn}
A spectral triple $(\A,\HH,\D)$ is given by

1) A representation $\pi:\A\to\B(\HH)$ of a  $*$-algebra $\A$ on the
Hilbert space
$\HH$.

2) A  self-adjoint (unbounded, densely defined)  operator 
$\D:{\rm dom}\,\D\to\HH$
such that
$[\D,\pi(a)]$ extends to a bounded operator on $\HH$ for all $a\in\A$ and
$\pi(a)(1+\D^2)^{-\frac{1}{2}}$ is compact for all $a\in\A$.

The triple is said to be even if there is an operator $\Gamma=\Gamma^*$ such
that
$\Gamma^2=1$, $[\Gamma,\pi(a)]=0$ for all $a\in\A$ and $\Gamma\D+\D\Gamma=0$
(i.e.
$\Gamma$ is a ${\Z}_2$-grading such that $\D$ is odd and $\pi(\A)$ is even.)
Otherwise the
triple is called odd.
\end{defn}

\begin{rem}
We will systematically omit the representation $\pi$ in future.
\end{rem}

Our aim is to define and study generalisations of 
spectral triples which model pseudo-Riemannian manifolds.
The interesting situation in this setting is when the underlying space is noncompact.
In order to discuss summability in this context, we recall a few definitions and results from 
\cite{CGRS2}, where a general definition of summability in the nonunital/noncompact context was
developed.

\begin{defn} Let  $\D$ be a densely defined self-adjoint operator on the Hilbert space $\HH$. 
Then for each $p\geq 1$ and $s>p$ we define a weight $\vf_s$ on $\B(\HH)$ by
$$
\vf_s(T):={\rm Trace}((1+\D^2)^{-s/4}T(1+\D^2)^{-s/4}),\quad 0\leq T\in \B(\HH),
$$
and the subspace $\B_2(\D,p)$ of $\B(\HH)$ by
$$
\B_2(\D,p):=\bigcap_{s>p} \Big({\rm dom}(\vf_s)^{1/2}\bigcap ({\rm dom}(\vf_s)^{1/2})^*\Big).
$$
The norms 
\begin{equation}\label{pn}
\B_2(\D,p)\ni T\mapsto \Q_n(T)
:=\left(\|T\|^2+\varphi_{p+1/n}(|T|^2)+\varphi_{p+1/n}(|T^*|^2)\right)^{1/2},\quad n=1,2,3\dots,
\end{equation}
take finite values on $\B_2(\D,p)$ and provide a topology on $\B_2(\D,p)$ 
stronger than the norm topology. 
We will always suppose that $\B_2(\D,p)$ has the topology defined by these norms.
\end{defn}

The space $\B_2(\D,p)$ is in fact a Fr\'echet algebra, \cite[Proposition 2.6]{CGRS2}, and plays the role of bounded square integrable operators. 
Next we introduce the bounded integrable operators.

On $\B_2(\D,p)^2$, the span of products $TS$, with $T,\,S\in \B_2(\D,p)$, define norms
\begin{equation}
\label{new-norm}
\PP_n(T)
:=\inf\Big\{\sum_{i=1}^k\Q_n( T_{1,i})\,\Q_n(T_{2,i})\ :\  
T=\sum_{i=1}^kT_{1,i}T_{2,i},\ \ \ T_{1,i},\,T_{2,i}\in \B_2(\D,p)\Big\}.
\end{equation}
Here the sums are finite and the infimum runs over all possible such representations of $T$. 
It is shown in \cite[page 13]{CGRS2} that the $\PP_n$ are norms on $\B_2(\D,p)^2$.

\begin{defn}
\label{def:sumpin-like-this}
Let $\mathcal{D}$ be a densely defined self-adjoint 
operator on $\HH$ and $p\geq 1$.
Define  $\B_1(\D,p)$ to be the completion of $\B_2(\D,p)^2$ 
with respect to the  
family of  
norms $\{\PP_n:\,n=1,2,\dots\}$. 
\end{defn}

\begin{defn}\label{def:spec-dim}
A   spectral triple  $(\A,\HH,\D)$, is said to be finitely summable if
there exists $s>0$ such that for all $a\in\A$, $a(1+\D^2)^{-s/2}\in\cl^1(\HH)$.
In such a case, we  let
$$
p:=\inf\big\{s>0\,:\,\forall a\in\A, \,\, {\rm Trace}\big(|a|(1+\D^2)^{-s/2}\big)<\infty
\big\},
$$
and call $p$ the spectral dimension of $(\A,\HH,\D)$.
\end{defn}

\begin{rems}
For the definition of the spectral dimension above to be  meaningful, one needs  two facts.
First, if $\A$ is the algebra of a finitely summable spectral triple,
we have $|a|(1+\D^2)^{-s/2}\in\cl^1(\HH)$ for all $a\in\A$, which follows by 
using the polar decomposition $a=v|a|$ and writing 
$$
|a|(1+\D^2)^{-s/2}=v^*a(1+\D^2)^{-s/2}.
$$
Observe that we are {\em not} asserting that $|a|\in\A$, which is typically not true in examples.
The second fact we require is that
$ {\rm Trace}\big(a(1+\D^2)^{-s/2}\big)\geq0$ for $a\geq 0$, which follows from
\cite[Theorem 3]{Bik}. 
\end{rems}

In contrast to the unital case, checking the finite summability
condition for a nonunital spectral triple can be difficult. This is because our definition
relies on control of the trace norm 
of the non-self-adjoint
operators  $a(1+\D^2)^{-s/2}$, $a\in\A$. It is shown in 
\cite[Propositions 3.16, 3.17]{CGRS2} that for a spectral triple $(\A,\HH,\D)$
to be finitely summable with spectral dimension $p$, it is necessary that $\A\subset \B_1(\D,p)$
and this condition is {\em almost} sufficient as well.

Anticipating the pseudodifferential calculus, we introduce subalgebras 
of $\B_1(\D,p)$ which `see' smoothness as well as summability. There
are several operators naturally associated to our notions of smoothness.

\begin{defn}
\label{parup}
Let $\mathcal{D}$ be a densely defined self-adjoint 
operator on the Hilbert space $\HH$.
Set $\HH_\infty=\bigcap_{k\geq 0} {\rm dom}\,\D^k$. For an 
operator $T\in\B(\HH)$ such that $T:\HH_\infty\to\HH_\infty$
we set
$$
\delta(T):=[(1+\D^2)^{1/2},T] .
$$
In addition, we recursively set   
$$
T^{(n)}:=[\D^2,T^{(n-1)}],\; n=1,2,\dots\quad \mbox{and}\quad T^{(0)}:=T.
$$ 
Finally, let 
\begin{align}
\label{LR}
L(T):=(1+\mathcal{D}^2)^{-1/2}[\mathcal{D}^2,T],\quad R(T):=[\mathcal{D}^2,T](1+\mathcal{D}^2)^{-1/2}.
\end{align}
\end{defn}

It is  proven in \cite{CM,C4} and \cite[Proposition 6.5]{CPRS2} that
\begin{equation}
\bigcap_{n\geq 0}{\rm dom}\,{L}^n
=\bigcap_{n\geq 0}{\rm dom}\,{R}^n\,=\bigcap_{k,\,l\geq 0}{\rm dom}\,L^k\circ R^l=\bigcap_{n\in\N}{\rm dom}\,{\delta}^n.
\label{eq:LR}
\end{equation}
%

To define $\delta^k(T)$ for $T\in\B(\HH)$ we need $T:\HH_l\to\HH_l$ for each $l=1,\dots,k$.
This is necessary before discussing boundedness of $\delta^k(T)$.

\begin{defn} 
\label{def-Bp}
 Let  $\D$ be a densely defined self-adjoint operator on the Hilbert space $\HH$, and $p\geq 1$.
 Then
 define for $k=0,1,2,\dots$
\begin{align*}
\B_1^k(\D,p)
&:=\big\{T\in\B(\HH)\,:\,T:\HH_l\to\HH_l,\ \mbox{and}
\ \  \delta^l(T)\in \B_1(\D,p)\ \ \forall\, l=0,\dots,k\big\}.
\end{align*}
Also set
$$
\B_1^\infty(\D,p):=\bigcap_{k=0}^\infty \B_1^k(\D,p).
$$

We equip $\B_1^k(\D,p)$, $k=0,1,2,\dots,\infty$, 
with the topology determined by the seminorms $\PP_{n,l}$ defined by
$$
\B(\HH)\ni T\mapsto \PP_{n,l}(T)
:=\sum_{j=0}^l\PP_n(\delta^j(T)),
\quad n=1,2,\dots,\quad l=0,\dots,k.
$$
\end{defn}

\begin{defn}
\label{delta-phi}
Let $(\A,\HH,\D)$ be a   spectral triple. Then we say that $(\A,\HH,\D)$ is
$QC^k$ summable if $(\A,\HH,\D)$ is finitely summable with spectral dimension $p$ and 
$$
\A\cup[\D,\A]\subset \B_1^k(\D,p).
$$
We say that $(\A,\HH,\D)$ is smoothly summable if it is $QC^k$ summable for all $k\in\N$ or, equivalently, 
if 
$$
\A\cup[\D,\A]\subset \B_1^\infty(\D,p).
$$
If $(\A,\HH,\D)$ is smoothly summable with spectral dimension $p$,
the $\delta$-$\varphi$-topology on $\A$ is determined by the family of norms 
$$
\A\ni a\mapsto \PP_{n,k}(a)+\PP_{n,k}([\D,a]), \,\,n=1,2,\dots,\,\,k=0,1,2,\dots,
$$ 
where the norms $\PP_{n,k}$ are those of Definition \ref{def-Bp}.
\end{defn}

In \cite{CGRS2} it was shown that the local index formula holds for smoothly summable spectral triples. Examples
show, \cite[page 45]{CGRS2}, that we can have $\A\subset \B_1^\infty(\D,p)$, while $[\D,\A]\not\in \B_1^\infty(\D,p)$. 
These examples show that we need assumptions to control the summability of $[\D,a]$ as well as $a\in\A$.


Next we recall the version of pseudodifferential calculus for nonunital
spectral triples developed in \cite{CGRS2}.

\begin{defn}
\label{op0}
Let  $\D$ be a densely defined self-adjoint operator on the Hilbert space $\HH$ and  $p\geq 1$.
The set of
{\bf  order-$r$ tame pseudodifferential operators} associated with 
$(\HH,\D)$ and $p\geq 1$ is given by
$$
{\rm OP}^r_0(\D):=(1+\D^2)^{r/2}\B_1^\infty(\D,p) ,\quad r\in\mathbb R,\qquad {\rm OP}^*_0(\D)
:=\bigcup_{r\in\R}{\rm OP}^r_0(\D).
$$

We topologise  ${\rm OP}^r_0(\D)$  with the family of norms 
$$
\PP_{n,l}^r(T):=\PP_{n,l}\big((1+\D^2)^{-r/2}T\big), \quad n,l\in\mathbb N,
$$
where the norms $\PP_{n,l}$ are as in Definition \ref{delta-phi}.
\end{defn}

\begin{rem}
To lighten the notation, we do not make explicit the 
important dependence on the real number $p\geq 1$
in the definition of the tame pseudodifferential operators. We also observe that  the definition
of all the spaces $\B_2(\D,p),\,\B_1(\D,p)$ and ${\rm OP}^r(\D)$  depends only on  $(1+\D^2)^{1/2}$.
\end{rem}

Since
$\B_1^\infty(\D,p)$ is {\it a priori} a nonunital algebra, functions of $\D$ alone do not
belong to ${\rm OP}^*_0$. In particular, not all `differential operators',
such as powers of $\D$, are  tame pseudodifferential operators.

\begin{defn}
Let  $\D$ be a densely defined self-adjoint operator on the Hilbert space $\HH$ and $p\geq 1$.
The set of {\bf regular order-$r$ pseudodifferential operators} is
$$
{{\rm OP}^r}(\D):=
(1+\D^2)^{r/2}\Big(\bigcap_{n\in\mathbb N}{\rm dom}\,\delta^n\Big),\quad r\in\mathbb R, 
\qquad {\rm OP}^*(\D):=\bigcup_{r\in\R}{\rm OP}^r(\D).
$$

The natural topology of ${\rm OP}^r(\D)$  is associated with the family of norms
$$
\sum_{k=0}^l\Vert \delta^k((1+\D^2)^{-r/2}T)\Vert, \quad l\in\mathbb N.
$$
\end{defn}

With this definition, 
${\rm OP}^r_0(\D)$ is a Fr\'echet space and both ${\rm OP}^0(\D)$ and ${\rm OP}^0_0(\D)$ are  Fr\'echet $*$-algebras. 
Moreover it is proved in \cite{CGRS2} that
for $r>p$, we have ${\rm OP}^{-r}_0(\D)\subset \cl^1(\HH)$ and that for all $r,\,t\in\mathbb R$ we have 
${\rm OP}^r_0(\D)\,{\rm OP}^t(\D),\ {\rm OP}^t(\D)\,{\rm OP}^r_0(\D)\subset{\rm OP}^{r+t}_0(\D)$. Thus
the tame pseudodifferential operators form an ideal within the regular pseudodifferential operators.

The other main ingredient of the pseudodifferential calculus is the one (complex) parameter group
\begin{equation}
\label{sigma}
\sigma^z(T):=(1+\D^2)^{z/2}\,T\,(1+\D^2)^{-z/2},\quad z\in\C,\ T\in {\rm OP}^*(\D).
\end{equation}
The group $\sigma$ restricts to a strongly continuous one 
parameter group on each ${\rm OP}^r(\D)$ and ${\rm OP}^r_0(\D)$,
and has a Taylor expansion, as first proved in \cite{CM}. 
The following version of the Taylor expansion is adapted to
tame pseudodifferential operators, and can be found in \cite[Proposition 2.35]{CGRS2}.

\begin{prop}
\label{TE}
Let  $\D$ be a densely defined self-adjoint operator on the Hilbert space $\HH$ and $p\geq 1$.
Let $T\in{\rm OP}^r_0(\D)$ and $z=n+1-\alpha$ with $n\in\N$ and $\Re(\alpha)\in(0,1)$. Then we have
$$
\s^{2z}(T)-\sum_{k=0}^nC_k(z)\,(\s^2-{\rm Id})^k(T)\in {\rm OP}^{r-n-1}_0(\D)\quad{\rm with}
\quad C_k(z):=\frac{z(z-1)\cdots(z-k+1)}{k!}.
$$
\end{prop}

We also prove here a lemma, which will be useful to us later on.
\begin{lemma}
\label{lem:OP-invert}
Let $\D$ be a densely defined self-adjoint operator on the 
Hilbert space $\HH$ and $p\geq1$. 
Let $T\in{\rm OP^0}(\D)$ have bounded inverse. Then $T^{-1}\in{\rm OP}^{0}(\D)$. 
\end{lemma}
\begin{proof}
We need to show that 
$\delta^n(T^{-1})$ is bounded for all $n\geq1$. We first check that
$$
\delta(T^{-1}) = -T^{-1} \delta(T) T^{-1} 
$$
is given by a product of bounded operators. Iterating this formula shows 
that there are combinatorial constants $C_{l,n,k}$ such that
\begin{align*}
\delta^n(T^{-1}) 
= \sum_{1\leq l\leq n} \sum_{1\leq k_1,\ldots,k_l\leq n,|k|=n} 
C_{l,n,k} T^{-1} \delta^{k_1}(T) T^{-1} \delta^{k_2} T^{-1} \cdots T^{-1} \delta^{k_l}(T) T^{-1} .
\end{align*}
Since $\delta^k(T)\in{\rm OP}^0(\D)$ is bounded for all $k$, we see that 
$\delta^n(T^{-1})$ is 
bounded for all $n$. Hence $T^{-1}$ is indeed an element of ${\rm OP}^0(\D)$. 
\end{proof}

In \cite{CGRS2} it was shown that Dirac type operators on 
complete  Riemannian manifolds of bounded geometry
give rise to smoothly summable spectral triples, with spectral 
dimension given by the dimension
of the underlying manifold. The bounded geometry 
hypothesis arises from the need to employ heat kernel
techniques to prove finite summability.

The next section explores the kind of analytic data that 
Dirac type operators on suitable pseudo-Riemannian manifolds give rise to. 
We use this section as motivation for the discussion of 
pseudo-Riemannian spectral triples to come. 

%
%
%


\section{Dirac operators on pseudo-Riemannian manifolds}
\label{sec:classical}

Let $(M,g)$ be an $n$-dimensional time- and space-oriented 
pseudo-Riemannian manifold of signature $(t,s)$. 
We will assume that we are given an orthogonal direct sum decomposition of the tangent bundle 
$TM = E_t\oplus E_s$ (with $\dim E_t = t$ and $\dim E_s = s$) such that the metric $g$ is 
negative definite on $E_t$ and positive definite on $E_s$. We shall consider elements of 
$E_t$ to be `purely timelike' and elements of $E_s$ to be `purely spacelike'. We shall 
denote by $T$  the projection $TM\rightarrow E_t$ onto the `purely timelike' subbundle. 
This projection is orthogonal with respect to $g$, which means that $g(v-Tv,Tw) = 0$ 
for all $v,w\in TM$. We then also have a \emph{spacelike reflection} $r := 1-2T$ which acts as 
$(-\1)\oplus\1$ on the decomposition $E_t\oplus E_s$. 

\begin{defn}
Given a direct sum decomposition $TM = E_t\oplus E_s$ (or, 
equivalently, given a spacelike reflection $r$, or a timelike projection $T$), 
we define a new metric $g_E$ on $TM$ by
\begin{align*}
g_E(v,w) &:= g(rv,w) .
\end{align*}
Since $r(1-T) = 1-T$, we readily check that $T$ is also an orthogonal 
projection with respect to the new metric $g_E$. Furthermore, $g_E$ is 
positive definite, and hence $(M,g_E)$ is a Riemannian manifold. Alternatively, we can also write
\begin{align*}
g(v,w) = g_E(rv,w) = g_E\big((1-T)v,(1-T)w\big) - g_E(Tv,Tw) .
\end{align*}
\end{defn}

\subsection{The Clifford algebras}

Let $\Cliff(TM,g_E)$ denote the real Clifford algebra with respect to $g_E$, and $\Cliff(TM,g)$ 
the real Clifford algebra with respect to $g$. Their complexifications are equal (since a complex 
Clifford algebra is independent of the signature of the chosen metric) and will be denoted 
$\CCliff(TM)$. Let the inclusion $TM \hookrightarrow \Cliff(TM,g_E)\subset\CCliff(TM)$ be 
denoted by $\gamma_E$. Our conventions are such that $\gamma_E(v)^*=-\gamma_E(v)$ 
and $\gamma_E(v)\gamma_E(w)+\gamma_E(w)\gamma_E(v) = -2g_E(v,w)$. We can now 
define a linear map $\gamma \colon TM \rightarrow \CCliff(TM)$ by
\begin{align}
\label{eq:defn_Cliff}
\gamma(v) := \gamma_E(v-Tv) - i \gamma_E(Tv) .
\end{align}
Since $T$ is a projection, we have $g_E(v-Tv,Tv) = 0$. We then verify that 
$\gamma(v)^2 = -g(v,v)$, since 
\begin{align*}
\gamma(v)^2 &= \gamma_E(v-Tv)^2 - \gamma_E(Tv)^2 - i\gamma_E(v-Tv)\gamma_E(Tv) 
- i\gamma_E(Tv)\gamma_E(v-Tv) \\
&= - g_E(v-Tv,v-Tv) + g_E(Tv,Tv) + 2i g_E(v-Tv,Tv) = -g(v,v) .
\end{align*}
Note that $(1-T-iT)^2 = r$. An alternative calculation for the complex-linear continuation 
$\gamma\colon TM\otimes\C\rightarrow\CCliff(TM)$ and the complexified metric then yields the same result:
\begin{align*}
\gamma(v)^2 = \gamma_E\big((1-T-iT)v\big)^2 = - g_E\big((1-T-iT)v,(1-T-iT)v\big) = - g_E(rv,v) = - g(v,v) .
\end{align*}
By the universal properties of Clifford algebras, this ensures that $\gamma$ extends to an algebra 
homomorphism $\gamma\colon \Cliff(TM,g)\rightarrow\CCliff(TM)$. Furthermore, by calculating
\begin{multline*}
\gamma(v)\gamma(w)^* = -\gamma_E(v-Tv)\gamma_E(w-Tw) - \gamma_E(Tv)\gamma_E(Tw) \\
- i\gamma_E(v-Tv)\gamma_E(Tw) + i\gamma_E(Tv)\gamma_E(w-Tw) ,
\end{multline*}
we find that $\gamma$ also satisfies 
\begin{align}
\label{eq:tilgamma_g}
\gamma(v)\gamma(w)^* + \gamma(w)^*\gamma(v) &= 2g_E(v,w) 
\end{align}
 with respect to the Riemannian metric.

\subsection{Dirac operators}

Take a local pseudo-orthonormal basis $\{e_i\}_{i=1}^{n}$ of $TM$ such that $\{e_i\}_{i=1}^{t}$ 
is a basis for $E_t$, and $\{e_i\}_{i=t+1}^{n}$ is a basis for $E_s$, i.e.
\begin{align*}
g(e_i,e_j) &= \delta_{ij}\kappa(j) , & \kappa(j) &= 
\begin{cases}-1 & j=1,\ldots,t;\\ 1 & j=t+1,\ldots,n. \end{cases}
\end{align*}
The basis $\{e_i\}_{i=1}^{n}$ is then automatically also an orthonormal basis with respect to $g_E$. 
We shall denote by 
$h$ the map $T^*M\rightarrow TM$ which maps
$\alpha\in T^*M$ to its dual in $TM$ with respect to the metric 
$g$. 
That is:
\begin{align*}
h(\alpha) &= v , & \Longleftrightarrow & & \alpha(w) &= g(v,w) \quad\text{for all } w\in TM .
\end{align*}
Let $\{\theta^i\}_{i=1}^n$ be the basis of $T^*M$ dual to $\{e_i\}_{i=1}^n$, so that 
$\theta^i(e_j) = \delta^i_j$. We then see that 
$h(\theta^i) = \kappa(i)e_i$. 

Fix a bundle $S\rightarrow M$, such that $\CCliff(TM) \simeq \End(S)$. We shall 
now discuss the construction of the standard Dirac operator on $\Gamma(S)$. Let us denote 
by 
$c$ 
the 
pseudo-Riemannian 
Clifford multiplication $\Gamma(T^*M\otimes S)\rightarrow \Gamma(S)$ given by 
\begin{align*}
c(\alpha\otimes\psi) &:= \gamma\big(h(\alpha)\big)\psi .
\end{align*}
Let $\nabla$ be the Levi-Civita connection for the pseudo-Riemannian metric $g$, and let $\nabla^S$ be its lift to the spinor bundle. 
We shall consider the Dirac operator on $\Gamma(S)$ defined as the composition
\begin{align*}
\sD &\colon \Gamma(S) \xrightarrow{\nabla^S} \Gamma(T^*M\otimes S) \xrightarrow{c} \Gamma(S) , &
\sD &:= c \circ \nabla^S = \sum_{j=1}^n \kappa(j) \gamma(e_j) \nabla^S_{e_j} .
\end{align*}
%
The pseudo-Riemannian Dirac operator $\sD$ is closeable, but it will not be symmetric, and not even normal! However, in the next section we will employ the notion of a Krein space to show that this operator will (under suitable assumptions) be Krein-self-adjoint with respect to an indefinite scalar product. 

\subsection{Krein spaces}
\label{subsec:Krein}

Let us recall the definition and some basic properties of Krein spaces from \cite[page 100]{Bog}.

\begin{defn}[{\cite[\S V.1]{Bog}}]
A \emph{Krein space} is a vector space $V$ with an indefinite inner product 
$(\cdot,\cdot)$ that admits a fundamental decomposition of the form 
$V=V^+\oplus V^-$ into a positive-definite subspace $V^+$ and a 
negative-definite subspace $V^-$, where $V^+$ and $V^-$ are 
intrinsically complete, i.e.\ complete with respect to the norms $\Vert v\Vert_{V^\pm} := |(v,v)|^{1/2}$. 
\end{defn}

Let $P^\pm$ denote the projections on $V^\pm$. Then the operator 
$\mJ := P^+ - P^-$ is self-adjoint and unitary, and defines a positive-definite 
inner product $(\cdot,\cdot)_\mJ := (\cdot,\mJ\cdot)$. Such an operator $\mJ$ 
(which depends on the choice of fundamental decomposition) 
is called a \emph{fundamental symmetry} of the Krein space $V$. 

A decomposable, non-degenerate inner product space $V$ is a Krein space 
if and only if, for every fundamental symmetry $\mJ$, the inner 
product $(\cdot,\cdot)_\mJ$ turns $V$ into a Hilbert space. 

Let $T$ be a linear operator with dense domain in the Krein space $V$. 
Then the Krein-adjoint $T^+$ of $T$ is defined on the domain
\begin{align*}
\dom T^+ := \{v\in V \mid \exists w\in V \colon \forall u\in\dom T \colon (Tu,v) = (u,w) \} ,
\end{align*}
and we define $T^+v := w$. Let $\mJ$ be a fundamental symmetry of 
$V$, and let $T^*$ be the adjoint of $T$ with respect to the Hilbert space 
inner product $(\cdot,\cdot)_\mJ$. We then have the relation $T^+ = \mJ T^*\mJ$. 
An operator $T$ is Krein-self-adjoint if and only if $\mJ T$ and $T\mJ$ are 
self-adjoint with respect to $(\cdot,\cdot)_\mJ$. 

\subsubsection{The Krein space of spinors}

We shall now briefly describe the Krein space of spinors on a pseudo-Riemannian spin manifold $(M,g)$. For more details we refer to \cite[\S3.3]{Bau}.

Recall the spacelike 
reflection $r$ which acts as $(-\1)\oplus\1$ on the decomposition $E_t\oplus E_s$ of 
the tangent bundle $TM$. Let $\Gamma_c(S)$ be the space of compactly supported smooth sections on the spinor bundle $S\to M$. There exists a positive-definite hermitian product (depending on the decomposition $E_t\oplus E_s$)
$$
(\cdot,\cdot)\colon \Gamma_c(S)\times\Gamma_c(S)\to C^\infty_c(M) ,
$$
which gives rise to the inner product
\begin{align*}
\la\psi_1,\psi_2\ra &:= \int_M (\psi_1,\psi_2) \nu_g , & \text{for } \psi_1,\psi_2\in\Gamma_c(S) ,
\end{align*}
where $\nu_g$ is the canonical volume form of $(M,g)$. 
The completion of $\Gamma_c(S)$ with respect to this inner product is denoted $L^2(M,S)$. 

As before, we have a (pseudo)-orthonormal basis 
$\{e_i\}_{i=1}^n$ such that $\{e_i\}_{i=1}^t$ is a basis for $E_t$ and $\{e_i\}_{i=t+1}^n$ 
is a basis for $E_s$. We shall define an operator $\mJ_M$ on $L^2(M,S)$ by
\begin{align*}
\mJ_M := i^{t(t-1)/2} \gamma(e_1)\cdots\gamma(e_t) .
\end{align*}
This operator $\mJ_M$ is a self-adjoint unitary operator, i.e.\ $\mJ_M^* = \mJ_M$ and 
$\mJ_M^2 = 1$. Furthermore, for the spacelike reflection $r$ as above and for a 
vector $v\in TM$, we have the relation
\begin{align}
\label{eq:spin_reflection}
\mJ_M\gamma(v)\mJ_M = (-1)^t \gamma(rv) .
\end{align}
We can now define an indefinite hermitian product
\begin{align*}
(\cdot,\cdot)_{\mJ_M} := (\mJ_M\cdot,\cdot) .
\end{align*}
Then $L^2(M,S)$ with indefinite inner product 
$
\la\cdot,\cdot\ra_{\mJ_M} := \int_M (\cdot,\cdot)_{\mJ_M} \nu_g
$
is a \emph{Krein space} with \emph{fundamental symmetry} $\mJ_M$. 

Let us quote a few facts from \cite[\S3.3]{Bau}. 
Using the spacelike reflection $r$ we obtain a new connection $\nabla^r := r\circ\nabla\circ r$, and its spin lift is given by $\nabla^{r,S} = \mJ_M \nabla^S \mJ_M$. 
For the Dirac operator $\sD$ we denote its (formal) Hilbert space adjoint (with respect to $\la\cdot,\cdot\ra$) as $\sD^*$ and its (formal) Krein space adjoint (with respect to $\la\cdot,\cdot\ra_{\mJ_M}$) as $\sD^+$. 
We shall define its real and imaginary parts as $\Re \sD := \frac12 (\sD+\sD^*)$ and $\Im \sD := -\frac i2 (\sD-\sD^*)$. 

\begin{prop}[{\cite[Satz 3.17]{Bau}}]
The formal adjoints $\sD^*$ and $\sD^+$ are locally of the form
\begin{align*}
\sD^+ &= (-1)^t \sD , & 
\sD^* &= \sum_{j=1}^n \gamma(e_j) \nabla^{r,S}_{e_j} , 
\end{align*}
where $(e_1,\ldots,e_n)$ is a local (pseudo)-orthonormal frame. 
\end{prop}

\begin{prop}[{\cite[Satz 3.19]{Bau}}]
\label{prop:Baum}
Let $(M,g)$ be an $n$-dimensional time- and space-oriented 
pseudo-Riemannian spin manifold of signature $(t,s)$. 
Let $r$ be a spacelike reflection, such that the associated Riemannian metric $g_E$ is complete. 
Consider the Dirac operator $\sD := c\circ\nabla^S$. Then 
\begin{enumerate}
\item the operators $\Re\sD$ and $\Im\sD$ are essentially self-adjoint with respect to $\la\cdot,\cdot\ra$; and 
\item the operator $i^t\sD$ is essentially Krein-self-adjoint with respect to $\la\cdot,\cdot\ra_{\mJ_M}$.
\end{enumerate}
\end{prop}

In fact the sum 
$$
\sD_E := \Re\sD + \Im\sD
$$
is also essentially self-adjoint 
self-adjoint (see Corollary \ref{cor:D_E-sa}), so to a Dirac operator $\sD$ we can associate a self-adjoint operator $\sD_E$. Its square is given by 
\begin{align*}
\sD_E^2 &= \big(\Re\sD\big)^2 + \big(\Im\sD\big)^2 + \big(\Re\sD\big)\big(\Im\sD\big)+\big(\Im\sD\big)\big(\Re\sD\big) 
= \frac12\left(\sD\sD^*+\sD^*\sD\right) - \frac i2\left(\sD^2-\sD^{*2}\right) .
\end{align*}
We shall write this as $\sD_E^2 = \la\sD\ra^2 - \pertsD$, where we define
\begin{align*}
\la\sD\ra^2 &:= \frac12\left(\sD\sD^*+\sD^*\sD\right) , & \pertsD &:= \frac i2\left(\sD^2-\sD^{*2}\right) .
\end{align*}

\begin{rem}
We will refer to the map $\sD\mapsto\sD_E$ as \emph{Wick rotation} of the Dirac operator. Let us justify this terminology by considering the basic example of a product space-time $M = \R\times N$, where $\R$ represents time and $N$ is an $n-1$-dimensional Riemannian manifold with metric $g_N$. Consider the pseudo-Riemannian metric on $M$ given by $g_M(t,\vec x) := \diag(-f(t),g_N(\vec x))$ for a strictly positive function $f\colon\R\to\R^+$. Let $\nabla$ be the Levi-Civita connection associated to $g_M$. Then the product-form of $g_M$ implies that $\nabla^r=\nabla$, and hence that the Dirac operator $\sD = \sum_{j=1}^n \kappa(j) \gamma(e_j) \nabla^S_{e_j}$ yields the Wick rotation $\sD_E = \sum_{j=1}^n \gamma_E(e_j) \nabla^S_{e_j}$. In this case, the map $\sD\mapsto\sD_E$ is therefore simply achieved through replacing $\kappa(j)\gamma(e_j)$ by $\gamma_E(e_j)$, which indeed resembles a Wick rotation of the time-coordinate. 
\end{rem}

\begin{lemma}
\label{lemma:sa}
Let $\sD$ be as in Proposition \ref{prop:Baum}.  
Then the operator $\la\sD\ra^2$ defined on the domain $\Dom \la\sD\ra^2 := \Dom \sD\sD^*\cap\Dom\sD^*\sD$ is essentially self-adjoint, elliptic and commutes with the fundamental symmetry $\mJ_M$. Hence it is also essentially Krein-self-adjoint. 
\end{lemma}
\begin{proof}
Consider 
$$
T:=\left(\sD,\sD^*\right):\,\dom\sD\cap\dom\sD^*\subset L^2(M,S)\to L^2(M,S)\oplus L^2(M,S).
$$
Then the operator 
$$
T^*T=\overline{\sD\sD^*+\sD^*\sD} = 2\overline{\la\sD\ra^2}
$$
has principal symbol $2g_E(\xi,\xi)$ (which easily follows using \eqref{eq:tilgamma_g}) and is therefore elliptic. Hence $T$ must also be elliptic, and then $T^*T$ is densely defined and essentially self-adjoint by \cite[Corollary 2.10]{BrMSh} (noting that the metric $g^{TM}$ defined therein is equal to $2g_E$, which is complete by assumption). By Krein-selfadjointness of $i^t\sD$ (Proposition \ref{prop:Baum}) we know that $\sD^* = (-1)^t \mJ_M \sD \mJ_M$, from which it easily follows that $\big[\la\sD\ra^2,\mJ_M\big] = 0$. 
\end{proof}

\subsection{The triple of a pseudo-Riemannian manifold}

To a Riemannian spin manifold $(M,g_E)$ one associates the spectral triple 
$\big(C_c^\infty(M),L^2(M,S), \sD\big)$,\footnote{Note that in fact, one should 
consider here the closure $\bar\sD$ of $\sD$, but to simplify notation we will 
just write $\sD$ for its closure as well.} where $\sD$ is the standard Dirac operator corresponding to the Riemannian metric $g_E$. 
In a similar manner, we associate to the pseudo-Riemannian manifold $(M,g)$ 
the triple $\big(C_c^\infty(M), L^2(M,S), \sD\big)$, where now $\sD$ is the standard Dirac operator defined with respect to the pseudo-Riemannian metric $g$, as considered above. 
The following few results derive
some basic properties of this triple which will motivate our abstract definition of a pseudo-Riemannian 
spectral triple.

To show that we obtain a spectral triple using the operator $\D_E$, we need to assume that the
Riemannian metric $g_E$ is well-behaved. 
Recall that  the injectivity radius  $r_{\rm inj}\in[0,\infty)$ of the Riemannian manifold $(M,g_E)$ 
is defined as
$$
r_{\rm inj}:=\inf_{x\in M}\sup\{r_x>0\},
$$
where $r_x$ is such that the exponential map  $\exp_x$ (defined w.r.t.\ the Riemannian metric $g_E$) is a 
diffeomorphism from  the ball $B(0,r_x)\subset T_xM$ to
an open neighborhood $U_{x}$  of $x\in M$. We observe that there is a related notion of injectivity radius
for Lorentzian manifolds which is adapted to the pair of metrics $(g,g_E)$; see \cite{CLF}.

\begin{defn}
\label{BG}
Let $(M,g)$ be an $n$-dimensional time- and space-oriented pseudo-Riemannian spin manifold of signature $(t,s)$. Let $r$ be a spacelike reflection such that $(M,g_E)$ is complete. We say that $(M,g,r)$ has bounded geometry if $(M,g_E)$ has strictly positive injectivity radius, and all the covariant derivatives of the (pseudo-Riemannian) curvature tensor of $(M,g)$ are bounded (w.r.t.\ $g_E$) on $M$.
A Dirac bundle on $M$ is said to have bounded geometry if in addition all the  covariant derivatives  
of $\Omega^S$, the curvature tensor of the connection $\nabla^S$, are bounded (w.r.t $g_E$) on $M$. 
For brevity, we simply say that $(M,g,r,S)$ has bounded geometry.
\end{defn}
A differential operator is said to have \emph{uniform $C^\infty$-bounded coefficients}, if for any atlas consisting of charts of normal coordinates, the derivatives of all order of the coefficients are bounded on the chart domain and the bounds are uniform on the atlas. It is shown in \cite[Propositions 5.4 \& 5.5]{R} that the assumption of bounded geometry is equivalent to the existence of a \emph{good coordinate ball}, that is a ball $B$ with center $0$ in $\R^n$ which is the domain of a normal coordinate system at every point of $M$, such that the Christoffel symbols of $\nabla$ and $\nabla^S$ lie in a bounded subset of the Fr\'echet space $C^\infty(B)$. Thus bounded geometry implies that the Dirac operator $\sD$ has uniform $C^\infty$-bounded coefficients. 

\begin{prop}
\label{prop:need-this}
Let $(M,g)$ be an $n$-dimensional time- and space-oriented pseudo-Riemannian spin manifold of signature $(t,s)$. 
Let $r$ be a spacelike reflection such that $(M,g_E)$ is complete and $(M,g,r,S)$ has bounded geometry. 
Then the triple $\big(C_c^\infty(M), L^2(M,S), \sD\big)$ satisfies 
\begin{itemize}
\item[1)] $\dom\langle\sD\rangle^2 := \dom\sD\sD^*\cap\dom\sD^*\sD$ is dense in 
$L^2(M,S)$ and $\langle\sD\rangle^2:=\frac12(\sD\sD^*+\sD^*\sD)$ is essentially self-adjoint on this domain;
\item[2a)] $(\sD^2-\sD^{*2})\in {\rm OP}^1(\langle\sD\rangle)$ and 
$[\langle\sD\rangle^2,\sD^2-\sD^{*2}]\in {\rm OP}^2(\langle\sD\rangle)$; 
\item[2b)] $a(\sD^2-\sD^{*2})(1+\langle\sD\rangle^2)^{-1}$ is compact for all $a\in C_c^\infty(M)$;
\item[3)] $[\sD,a]$ and $[\sD^*,a]$ extend to 
bounded operators on $\HH$ for all $a\in C_c^\infty(M)$ (in particular, 
all $a\in C_c^\infty(M)$ preserve $\dom\sD$, $\dom\sD^*$);
\item[4)] $a(1+\langle\sD\rangle^2)^{-1/2}\in\K\big(L^2(M,S)\big)$ for all $a\in C_c^\infty(M)$.
\end{itemize}
\end{prop}
\begin{proof}
From Lemma \ref{lemma:sa} we know that $\la\sD\ra^2$ is essentially self-adjoint, which proves 1), and elliptic, which proves 4).
Since $\sD,\,\sD^*$ are first order differential operators, 3) is immediate.
%
%
%
Consider the operator $\pertsD:=\frac i2(\sD^2-\sD^{*2})$, initially defined on the dense subset of compactly supported smooth sections $\Gamma_c(M,S)$. Since $\sD^{*2} = \mJ_M\sD^2\mJ_M$ by Proposition \ref{prop:Baum}, we see that $\sD^2-\sD^{*2} = [\sD^2,\mJ_M]\mJ_M$. As $\sD^2$ is a second-order differential operator whose principal symbol commutes with $\mJ_M$, we see that $\pertsD$ is a first-order differential operator. Hence $\pertsD (1+\la\sD\ra^2)^{-1/2}$ is a bounded operator. 
Ellipticity of $\la\sD\ra^2$ then implies 2b). 
Under the assumption of bounded geometry, $\pertsD (1+\la\sD\ra^2)^{-1/2}$ is in fact a smooth uniformly bounded operator with uniformly bounded covariant derivatives, which implies 2a). 
\end{proof}

\begin{rem}
So far we have only used the boundedness of the coefficients of $\sD$ and the
ellipticity of $\langle \sD\rangle$. Later we will use information from the bounded geometry hypothesis
to deduce summability properties for $\langle \sD\rangle$.
\end{rem}

\section{Pseudo-Riemannian spectral triples}\label{sRst}

\subsection{Preliminaries on unbounded operators}

Our definition of pseudo-Riemannian spectral triple relies on some 
results about unbounded operators. Here we recall some standard facts, and use these to prove some 
perturbation results that we will require.

We observe that questions about non-self-adjoint unbounded operators
often  revolve around the issue of `symmetric versus self-adjoint', and have a flavour of boundary value problems.

Here we will aim to control all `symmetric versus self-adjoint' issues, focusing instead on the 
`algebraic' lack of symmetry that we saw was typical of Dirac operators on pseudo-Riemannian manifolds.
Naturally we must make fairly strong assumptions to ensure that `symmetric versus self-adjoint' issues
do not interfere with the purely algebraic aspects. 

First recall that if $\D:\mbox{dom}\,\D\to\HH$ is a 
densely defined closed operator, then both $\D\D^*$ and $\D^*\D$ are densely defined 
(with obvious domains), self-adjoint and positive. 
The operator $\la\D\ra^2:=\frac{1}{2}(\D\D^*+\D^*\D)$ is defined on $\mbox{dom}\,\D^*\D\cap\mbox{dom}\,\D\D^*$, 
and when this is dense, $\la\D\ra^2$ is a densely defined symmetric operator. 
The symmetric operator $\la \D\ra^2$ is positive, since it is the operator associated with the quadratic form
$$
q(\phi,\psi)=\la\D\phi,\D\psi\ra+\la\D^*\phi,\D^*\psi\ra,\ \ \ \ \phi,\,\psi\in 
\mbox{dom}\,\D\,\cap\,\mbox{dom}\,\D^*.
$$
Then, by \cite[Theorem X.23]{RS}, $q$ is a closeable form, and its closure 
is associated with a unique positive operator which we denote by 
$\la\D\ra^2_F$, the Friedrich's extension. 

By considering instead $\frac{1}{2}(\D\D^*+\D^*\D)+1$, and applying \cite[Theorem X.26]{RS}, 
we see that $\frac{1}{2}(\D\D^*+\D^*\D)$ is essentially self-adjoint if and only if the 
Friedrich's extension is the only self-adjoint extension, which occurs if and only if the range 
of $\frac{1}{2}(\D\D^*+\D^*\D)+1$ is dense. The key assumption we shall make to avoid `symmetric versus self-adjoint' issues is 
that $\D\D^*+\D^*\D$ is essentially self-adjoint.

\begin{defn}
\label{lem:symm}
Given  $\D:{\rm dom}\,\D\to\HH$  a densely defined closed unbounded operator with  
${\rm dom}\,\D\,\cap\,{\rm dom}\,\D^*$ dense in $\HH$, we define the `Wick rotated' operator
$$
\D_E=\frac{e^{i\pi/4}}{\sqrt{2}}(\D-i\D^*)=\frac12(\D+\D^*)+\frac i2(\D-\D^*)
$$ 
with initial domain ${\rm dom}\,\D\,\cap\,{\rm dom}\,\D^*$. Then
$\D_E$ is symmetric and so closeable. On ${\rm dom}\,\D_E^2$ we have 
$$ 
\D_E^2=\frac{1}{2}(\D\D^*+\D^*\D)+\frac{i}{2}(\D^2-\D^{*2})=:\la\D\ra^2+\pert,
$$
where we write $\pert=\frac{i}{2}(\D^2-\D^{*2})$.
\end{defn}

%



The  next lemma appears as \cite[exercise 28, Chapter X]{RS} 
(see also the proof of \cite[Lemma 3]{B}).

\begin{lemma}
\label{lem:sqrt-self-adj}
Suppose that $T$ is a symmetric operator on the Hilbert space $\HH$ with ${\rm dom}\,T^2$ dense in $\HH$. If
$T^2$ is essentially self-adjoint on ${\rm dom}\,T^2$ then $T$ is essentially self-adjoint.
\end{lemma}

Our main technical tool for passing from pseudo-Riemannian spectral triples to spectral triples
is the commutator theorem, \cite[Theorem X.36]{RS}. We restate a slightly weaker version of this result using
the language of pseudodifferential operators from Section \ref{prelim}.

\begin{thm} 
\label{thm:RS}
Let $N\geq 1$ be a positive self-adjoint operator on the Hilbert space $\HH$. If $A\in {\rm OP}^2(N^{1/2})$ is closed and symmetric and
furthermore $[N,A]\in{\rm OP}^2(N^{1/2})$ then

1) ${\rm dom}\,N\subset {\rm dom}\,N+A$ and there exists a constant $C>0$ such that for all $\xi\in {\rm dom}\,N$
$$
\Vert (N+A)\xi\Vert\leq C\,\Vert N\xi\Vert;
$$
2) the operator $N+A$ is essentially self-adjoint on any core for $N$.
\end{thm}

\subsection{The definition and immediate consequences}\label{subsec:defs}

The following definitions are intended to give an analytic framework
for
dealing with the non-normal unbounded operators arising in
pseudo-Riemannian
geometry. We have endeavoured to make this framework as general as
possible, and it is easy to see from Proposition \ref{prop:need-this} that
the triples we constructed from (suitable) pseudo-Riemannian manifolds 
satisfy this definition.
After presenting definitions of smoothness and  summability 
by
analogy with spectral triples, we give a  broad range of examples.

Recall that given a closed densely defined operator $\D:{\rm dom}\,\D\subset\HH\to\HH$, we define 
$\la\D\ra^2=\frac{1}{2}(\D\D^*+\D^*\D)$ and $\pert=\frac{i}{2}(\D^2-\D^{*2})$
when defined. We also set $\HH_\infty=\cap_{k\geq 0}\,{\rm \dom}\,\la\D\ra^k$.

\begin{defn}\label{no-beta-defn} A pseudo-Riemannian spectral triple
$(\A,\HH,\D)$ consists of  a $*$-algebra $\A$ 
represented on the Hilbert space $\HH$
as bounded operators, along with a densely defined closed operator 
$\D:{\rm dom}\,\D\subset\HH\to\HH$ such that

1) ${\rm dom}\,\D^*\D\cap{\rm dom}\,\D\D^*$ is dense in $\HH$ 
 and $\la\D\ra^2$ is essentially self-adjoint on this domain;
 
 2a) $\pert:\HH_\infty\to\HH_\infty$, and  $[\la\D\ra^2,\pert]\in {\rm OP}^2(\la\D\ra)$;
 
2b) 
$a\pert(1+\la\D\ra^2)^{-1}$ is compact for all $a\in\A$;

3)  $[\D,a]$ and $[\D^*,a]$ extend to  bounded operators on $\HH$ for all $a\in\A$ (in 
particular, all $a\in\A$ preserve ${\rm dom}\,\D$, ${\rm dom}\,\D^*$);

4) $a(1+\la\D\ra^2)^{-1/2}\in\K(\HH)$ for all $a\in\A$;

The pseudo-Riemannian triple is said to be even if there exists 
$\Gamma\in\B(\HH)$
such that $\Gamma=\Gamma^*$, $\Gamma^2=1$, $\Gamma a=a\Gamma$ for all $a\in\A$
and $\Gamma\D+\D\Gamma=0$. Otherwise the pseudo-Riemannian triple is said to be
odd.
\end{defn}

The definition implies that the operator $\pert$ is in fact in ${\rm OP}^2(\la\D\ra)$.

\begin{lemma}
\label{lem:bound_A}
Let $(\A,\HH,\D)$ be a pseudo-Riemannian spectral triple. Then
$$
(1+\la\D\ra^2)^{-1/2}\pert(1+\la\D\ra^2)^{-1/2}\in {\rm OP}^0(\la\D\ra).
$$
Hence $\pert\in {\rm OP}^2(\la\D\ra)$.
\end{lemma}
\begin{proof}
The operators $\D\D^*$ and $\D^*\D$ are positive 
and bounded by $1+\la\D\ra^2$, from which it follows that
$$
\|\D(1+\la\D\ra^2)^{-1/2}\|^2 = \|(1+\la\D\ra^2)^{-1/2}\D^*\D(1+\la\D\ra^2)^{-1/2}\| \leq 2 .\\
$$
Hence $\D$ (and similarly $\D^*$) is bounded by $(1+\la\D\ra^2)^{1/2}$. 
Thus for $\pert = \frac i2 (\D^2-\D^{*2})$ we obtain
\begin{multline*}
\|(1+\la\D\ra^2)^{-1/2}\pert(1+\la\D\ra^2)^{-1/2}\| \\
\leq \frac12 \|(1+\la\D\ra^2)^{-1/2}\D^2(1+\la\D\ra^2)^{-1/2}\| + \frac12 \|(1+\la\D\ra^2)^{-1/2}\D^{*2}(1+\la\D\ra^2)^{-1/2}\| 
\leq 2 .
\end{multline*}
Applying $L_{\la\D\ra}=(1+\la\D\ra^2)^{-1/2}[\la\D\ra^2,\cdot]$ repeatedly, using Equation \eqref{eq:LR}, 
and recalling that $\pert^{(n)}=[\la\D\ra^2,\pert^{(n-1)}]\in {\rm OP}^{n+1}(\la\D\ra)$ for all
$n\geq 1$ from part 2a) of Definition \ref{no-beta-defn}, we see that
$$
L^n_{\la\D\ra}((1+\la\D\ra^2)^{-1/2}\pert(1+\la\D\ra^2)^{-1/2})
=(1+\la\D\ra^2)^{-1/2}L^n_{\la\D\ra}(\pert)(1+\la\D\ra^2)^{-1/2}
$$
is bounded for all $n\geq 1$. 
\end{proof}

\begin{rems} 
In Definition \ref{no-beta-defn}, both parts of condition $2)$  
are intended to force $\D^2-\D^{*2}$ to be a 
`first order operator', regarding $\la\D\ra^2$ as second order. 
There are two things to control here: the order
of the `differential operators' appearing in $\D^2-\D^{*2}$, 
and the growth of the `coefficients'. All of these
quotation marks can be understood quite literally in the 
classical case described in the last section. 

If we were to restrict attention, in the classical case, 
to differential operators with bounded coefficients, 
we would expect the easiest assumptions to force 
$\D^2-\D^{*2}$ to be first order would be 
\begin{equation}
\D^2-\D^{*2}\in {\rm OP}^1(\la\D\ra)
\label{eq:easy-assumption1}
\end{equation}
and
\begin{equation}
a(\D^2-\D^{*2})\in {\rm OP}^1_0(\la\D\ra)\ \ \mbox{for all }a\in\A.
\label{eq:easy-assumption2}
\end{equation}
In fact, Equation \eqref{eq:easy-assumption1} actually implies $2a)$, 
since ${\rm OP}^1(\la\D\ra)\subset {\rm OP}^2(\la\D\ra)$, 
while Equation \eqref{eq:easy-assumption2} together with $4)$
implies  $2b)$. 

The reason for weakening the assumptions so that {\em a priori} 
$\pert=\frac{i}{2}(\D^2-\D^{*2})\in {\rm OP}^2(\la\D\ra)$
only,  is to 
allow for unbounded coefficients and also to allow 
for non-smooth elements in our algebra. 
(For instance the condition \eqref{eq:easy-assumption2} forces $a\in\A$ to be smooth). 
The  harmonic oscillator, treated in detail later, is an example of where unbounded coefficients occur.

Ultimately we will restrict attention to smoothly summable 
pseudo-Riemannian spectral triples, but
we would like to have a definition with minimal smoothness requirements. 
We have found numerous not-quite-equivalent ways of 
specifying that $\D^2-\D^{*2}$ and/or $a(\D^2-\D^{*2})$
are order one operators, each with slightly different 
(smoothness) hypotheses. All these different approaches allow us to prove that we can
obtain a spectral triple from a pseudo-Riemannian spectral triple, though
the methods differ in each case. For instance, some of the other possible
assumptions allow  the use of the Kato-Rellich theorem 
in  place of the commutator theorem we employ here.
\end{rems}


\begin{cor}
\label{lem:self-adj2} 
Let $(\A,\HH,\D)$ be a pseudo-Riemannian spectral triple. Then the 
operator $\la\D\ra^2+\frac{i}{2}(\D^2-\D^{*2})$ is essentially self-adjoint on any core for $\la\D\ra^2$,
and ${\rm dom}\,\la\D\ra^2\subset{\rm dom}\,\la\D\ra^2+\frac{i}{2}(\D^2-\D^{*2})$.
Moreover there is $C>0$ such that for all $\xi\in{\rm dom}\,\la\D\ra^2$, we have
$$
\big\Vert \big(1+\la\D\ra^2+\frac{i}{2}(\D^2-\D^{*2})\big)\xi\big\Vert\leq C\Vert (1+\la\D\ra^2)\xi\Vert.
$$
Hence we also have
$$
\big\Vert \big(1+\la\D\ra^2+\frac{i}{2}(\D^2-\D^{*2})\big)(1+\la\D\ra^2)^{-1}\big\Vert\leq C.
$$
\end{cor}

\begin{proof} All of these statements follow from \cite[Theorem X.36]{RS}, quoted here as Theorem \ref{thm:RS},
whose hypotheses follow from 
the condition $2a)$. 
\end{proof}

%

\begin{cor}
\label{cor:D_E-sa}
Let $(\A,\HH,\D)$ be a pseudo-Riemannian spectral triple. Then the operator
$\D_E=\frac{e^{i\pi/4}}{\sqrt{2}}(\D-i\D^*)$ is essentially self-adjoint. Hence
$\D_E^2=\la\D\ra^2+\frac{i}{2}(\D^2-\D^{*2})$ is a positive operator.
\end{cor}

\begin{proof}
The first statement follows from Lemma \ref{lem:sqrt-self-adj}, while the second
is a consequence of the reality of the spectrum of $\D_E$.
\end{proof}

Ultimately we will be interested in obtaining spectral triples from pseudo-Riemannian spectral triples
which satisfy the hypotheses of the local index formula. For this we introduce the following notion of 
spectral dimension and smooth summability for pseudo-Riemannian spectral triples.

\begin{defn}\label{def:spec-dim-semi}
A  pseudo-Riemannian spectral triple  $(\A,\HH,\D)$, is said to be finitely summable if
there exists $s>0$ such that for all $a\in\A$, $a(1+\la\D\ra^2)^{-s/2}\in\cl^1(\HH)$.
In such a case, we  let
$$
p:=\inf\big\{s>0\,:\,\forall a\in\A, \,\, {\rm Trace}\big(|a|(1+\la\D\ra^2)^{-s/2}\big)<\infty
\big\},
$$
and call $p$ the spectral dimension of $(\A,\HH,\D)$.
\end{defn}


\begin{defn}
\label{delta-phi-semi}
Let $(\A,\HH,\D)$ be a pseudo-Riemannian  spectral triple. Define the set $S^0 := \A\cup[\D,\A]\cup[\D^*,\A]$, and then recursively define the sets $S^n := [\la\D\ra^2,S^{n-1}]\cup[\pert,S^{n-1}]$ for $n\geq1$. Then  $(\A,\HH,\D)$ is
$QC^k$ summable if $(\A,\HH,\D)$ is finitely summable with spectral dimension $p$ and 
\begin{align*}
S^n&\subset \B_1^{k-n}(\la\D\ra,p) (1+\la\D\ra^2)^{n/2} \ \ \forall\,0\leq n\leq k.
\end{align*}
We say that $(\A,\HH,\D)$ is smoothly summable if it is $QC^k$ summable for all $k\in\N$ or, equivalently, 
if 
\begin{align*}
S^n&\subset \B_1^\infty(\la\D\ra,p) (1+\la\D\ra^2)^{n/2} = {\rm OP}^n_0(\la\D\ra) \ \ \forall\,n\geq 0.
\end{align*}
\end{defn}

\begin{rem}
If $\pert=\frac{i}{2}(\D^2-\D^{*2})=0$, so that $\la\D\ra^2=\D_E^2$, then the  
definition of smooth summability here would reduce to $S^0\subset {\rm OP}^0_0(\la\D\ra)=
{\rm OP}^0_0(|\D_E|)$. 
\end{rem}

\subsection{Examples}
\subsubsection{Pseudo-Riemannian manifolds} Proposition \ref{prop:need-this} shows
that those pseudo-Riemannian manifolds arising by reflection from complete Riemannian manifolds of bounded
geometry have a Dirac operator satisfying our definition of pseudo-Riemannian spectral triples.

For finite summability, we observe that the bounded geometry hypothesis ensures that 
${\rm Trace}(a(1+\la\sD\ra^2)^{-s/2})$ is finite for $s>n=\dim M$, 
where $a\in C_c^\infty(M)$ is a compactly supported smooth function. 
Hence the spectral dimension $p$ is equal to the metric dimension $n$; see \cite{CGRS2}. 

For smooth summability, we also need ${\rm Trace}(\gamma(e_j) a(1+\la\sD\ra^2)^{-s/2})$ to be finite, which holds since $a$ is compactly supported, and so $\gamma(e_j) a$ is also compactly supported and bounded. 
Furthermore, we observe that $\la\sD\ra^2$ is a uniformly elliptic second order differential operator
with scalar principal symbol (the metric). 
Hence $\la\sD\ra^2$ determines the usual order of compactly supported pseudodifferential
operators. 

The correction $\pertsD=\frac{i}{2}(\sD^2-\sD^{*2})$ is a first order operator, and the bounded geometry
hypothesis implies that $\pertsD \in {\rm OP}^1(\la\sD\ra)$. 
Hence
$[\pertsD,{\rm OP}^r_0(\la\sD\ra)]\subset {\rm OP}^{r+1}_0(\la\sD\ra)$. It thus follows that the triple $(C_c^\infty(M), L^2(M,S), \sD)$ is a smoothly summable pseudo-Riemannian spectral triple.

\subsubsection{Previous work}

Paschke and Sitarz, \cite{PS}, used equivariance with respect to a group or quantum group of isometries 
to produce examples of Lorentzian spectral triples, which are a special case
of the definition presented above. In particular they constructed examples on noncommutative tori and noncommutative 
$3$-spheres (both isospectral deformations of $S^3$  and $SU_q(2)$).

Strohmaier and Van Suijlekom, \cite{S,vS}, used a Krein space formulation for their definition of 
semi-Riemannian spectral triples, but this is largely equivalent to the approach we
have adopted, as can be seen from Sections \ref{subsec:Krein} and \ref{LIT}. Strohmaier also examined noncommutative tori, while Van Suijlekom examined
deformations of generalised cylinders. All of these examples can be put in our framework.

\subsubsection{Finite Geometries}
\label{subsub:finite}
Just as there are virtually no constraints to the existence of a
spectral triple
for a finite dimensional algebra,  pseudo-Riemannian spectral
triples are easily constructed in this case.
So let $\A$ be a finite dimensional complex algebra, i.e. a
direct sum of simple
matrix algebras. Choose two representations of $\A$ on finite dimensional
Hilbert spaces $\HH_1$ and $\HH_2$. Let $\HH=\HH_1\oplus\HH_2$,
and $\Gamma=\bma
1 & 0\\ 0 &-1\ema$ with respect to this decomposition. All we need now is a
`Dirac' operator. Choose any linear map $B:\HH_1\to\HH_2$, and set
$\D=\bma 0 &
0\\ B&0\ema$. Then the definition of an even  pseudo-Riemannian spectral
triple
is
trivially fulfilled. Likewise it is trivially smoothly summable.

\subsubsection{First order differential operators}
\label{subsub:1st-order}
We consider a constant coefficient first order differential operator 
of the form
$$
\D=\sum_{j=1}^nM_j\frac{\p}{\p x^j}+K
$$
where $K,M_j\in M_d(\C)$. The operator $\D$, acting on the smooth compactly supported
sections in $L^2(\R^n,\C^d)$, extends to a
closed and densely defined operator. 
One may check that $(C_c^\infty(\R^n), L^2(\R^n,\C^d), \D)$ yields a pseudo-Riemannian spectral triple provided that for all $j,k=1,\ldots,n$ and for all $0\neq\xi=(\xi_1,\dots,\xi_n)\in\R^n$ the following three conditions hold:
\begin{gather*} 
\sum_{j,k=1}^n(M^*_jM_k+M_jM_k^*)\xi_j\xi_k \in GL_d(\C) ,\\
\{M_j,M_k\}=\{M_j^*,M_k^*\} ,\\
\left[\sum_{j,k=1}^n(M^*_jM_k+M_jM_k^*)\xi_j\xi_k,
\sum_{j=1}^n\left(\{M_j,K\}+\{M_j^*,K^*\}\right)\xi_j\right]=0.
\end{gather*}

Provided that in addition
\begin{align*}
[M_j,M_kM_l^*+M_k^*M_l]&=0, & \text{and} && [K,M_kM_l^*+M_k^*M_l]&=0,
\label{eq:smooth-comms}
\end{align*}
the pseudo-Riemannian spectral triple is in fact smoothly summable.

\subsubsection{The harmonic oscillator} 
We consider
(the creation operator associated to) the harmonic oscillator
$$ 
\D=\frac{d}{dx}+x,
$$
acting on $L^2(\R)$. To show that 
we obtain a pseudo-Riemannian spectral triple, we observe that the operator 
$$
\la\D\ra^2=-\frac{d^2}{dx^2}+x^2
$$
is well-known to be essentially self-adjoint and to have compact resolvent. For all $a\in C_1^\infty(\R)$,
the smooth functions with integrable derivatives, 
the commutators $[\D,a]$ and $[\D^*,a]$ are bounded. The other items to check
involve the operator
$$
\pert:=\frac{i}{2}(\D^2-\D^{*2})=i(1+2x\frac{d}{dx}) ,
$$
which is first order as a differential operator, but has an unbounded coefficient. 
With $\D_E=\frac12(\D+\D^*)+\frac i2(\D-\D^*)$ we have $\D_E=i\frac{d}{dx}+x$, but we could just
as easily take $\widetilde{\D_E}=\frac12(\D+\D^*)-\frac i2(\D-\D^*)$ to get 
$\widetilde{\D_E}=-i\frac{d}{dx}+x$. Our 
results apply to both possibilities.

\begin{lemma}
\label{lem:osc-one}
For all $n\geq 0$ the operator $\pert^{(n)}$ is an element of ${\rm OP}^2(\la\D\ra)$,
and for all $a\in C_1^\infty(\R)$ the operator $a\pert(1+\la\D\ra^2)^{-1}$ is compact.
\end{lemma}

\begin{proof}
Straightforward calculations show that
$$
\pert^{(1)}=[\la\D\ra^2,\pert]=4i\la\D\ra^2-8ix^2,\quad \pert^{(2)}=16\pert,
$$ 
and so it suffices to check the first claim for $n=0,1$. We begin by observing that
$
0\leq x^2\leq x^2-\frac{d^2}{dx^2}+1
$
implies 
\begin{equation}
\Vert x(1+x^2-\frac{d^2}{dx^2})^{-1/2}\Vert^2
=\Vert (1+x^2-\frac{d^2}{dx^2})^{-1/2}x\Vert^2
\leq 1.
\label{eq:order-1}
\end{equation}
Similarly $\Vert \frac{d}{dx}(1+x^2-\frac{d^2}{dx^2})^{-1/2}\Vert\leq 1$. Hence
$$
(1+\la\D\ra^2)^{-1/2}\pert(1+\la\D\ra^2)^{-1/2}\ \mbox{and}\ (1+\la\D\ra^2)^{-1/2}\pert^{(1)}(1+\la\D\ra^2)^{-1/2}
$$
are bounded operators. Hence $\pert\in {\rm OP}^2(\la\D\ra)$ as is $\pert^{(1)}$.
If $a$ is a bounded integrable function, 
the product $ax$ is bounded, and so by the compactness of $(1+\la\D\ra^2)^{-1/2}$
we see that $a\pert(1+\la\D\ra^2)^{-1}$ is compact. 
\end{proof}

Thus the harmonic oscillator gives rise to a pseudo-Riemannian spectral triple.

\begin{prop}
\label{prop:harm_osc_smooth}
The pseudo-Riemannian spectral triple 
$(C_1^\infty(\R),L^2(\R),\frac{d}{dx}+x)$ is smoothly summable
with spectral dimension $1$.
\end{prop}
\begin{proof}
Let $a:\R\to[0,\infty)$ be a smooth integrable function, and
let $(x,y)\mapsto k_t(x,y)$ be the integral kernel of $e^{-t\la\D\ra^2}$ for $\D=\frac{d}{dx}+x$.
Mehler's formula gives
$$
k_t(x,y)=\frac{1}{\sqrt{2\pi\sinh(2t)}}e^{-\frac{1}{2}\coth(2t)(x^2+y^2)+{\rm cosech}(2t)xy}.
$$
Then for $s>2$
\begin{align*}
{\rm Trace}(a(1+\la\D\ra^2)^{-s/2})
&=\frac{1}{\Gamma(s/2)}\int_\R a(x)\int_0^\infty t^{s/2-1}\,e^{-t}k_t(x,x)\,dt\,dx\\
&\leq \frac{1}{\Gamma(s/2)}\int_\R a(x)\,dx\,\int_0^\infty t^{s/2-1}\,e^{-t}\frac{1}{\sqrt{2\pi\sinh(2t)}}\,dt
\end{align*}
and this remains finite for $s>1$. Thus the spectral dimension is $\leq 1$. To see that the
spectral dimension is $\geq 1$, and so is precisely $1$, one computes this trace for the function
$a:x\mapsto e^{-x^2}$. 

For the algebra $\A=C_1^\infty(\R)$, 
we find that the commutators $[\D,a]$ and $[\D^*,a]$ are again elements of this algebra. 
With the notation of Definition \ref{delta-phi-semi}, we thus find that
 $S^0 := \A\cup[\D,\A]\cup[\D^*,\A] = \A$. 
The above computations now allow us to see that 
$S^0$ lies in $B_1(\la\D\ra,1)$. 


Finally, we need to check that $S^n \subset {\rm OP}^n_0(\la\D\ra)$ for all $n$. First, by Equation \eqref{eq:order-1}, and the relations
\begin{align*}
\left[\la\D\ra^2,\frac{d}{dx}\right] &= 2x , & \left[\la\D\ra^2,x\right] &= -2\frac{d}{dx} 
&[\pert,x^n\frac{d^m}{dx^m}]=2i(n-m)x^n\frac{d^m}{dx^m},
\end{align*}
we can see that both multiplication by $x$ and $\frac{d}{dx}$
lie in ${\rm OP}^1(\la\D\ra)$. For $a\in C_1^\infty(\R)$, we use the computation 
$$
[\la\D\ra^2,a]=-a''-2a'\frac{d}{dx}
$$
and a simple induction to see that $a\in {\rm OP}^0(\la\D\ra)$. 
Hence  $C_1^\infty(\R)\subset {\rm OP}^0_0(\la\D\ra)$.
It is then straightforward to check that any element 
$T\in S^n$ can be written in the form 
$$
T = \sum_{k+l\leq n} a_{k,l} x^k \frac{d^l}{d x^l}
$$
for functions $a_{k,l}\in C_1^\infty(\R)$. This is obviously true for $n=0$. 
Assuming it holds for all $T\in S^n$ for some $n$, one shows it also 
holds for $n+1$ by explicitly calculating the commutators 
$[\la\D\ra^2,T]$ and $[\pert,T]$. So it follows by induction that 
indeed $T = \sum_{k+l\leq n} a_{k,l} x^k \frac{d^l}{d x^l}$ for all $T\in S^n$, for all $n$. 
Since $a_{k,l}\in {\rm OP}^0_0(\la\D\ra)$, $x^k \in {\rm OP}^k(\la\D\ra)$, 
and $\frac{d^l}{d x^l} \in {\rm OP}^l(\la\D\ra)$, it follows that 
$S^n \subset {\rm OP}^n_0(\la\D\ra)$ for all $n$.
Therefore we conclude that the pseudo-Riemannian spectral triple 
$(C_1^\infty(\R),L^2(\R),\frac{d}{dx}+x)$ is smoothly summable. 
%
%
%
%
%
\end{proof}

\subsubsection{Semifinite examples} There is a notion of
semifinite spectral triple, \cite{BeF,CPRS1,CPRS2,CPRS3}, where
$(\B(\HH),\K(\HH),\mbox{Trace})$ are replaced by $(\cn,\K(\cn,\tau),\tau)$
where $\cn$ is an arbitrary semifinite von Neumann algebra, $\K(\cn,\tau)$
is the ideal of $\tau$-compact operators in $\cn$, and $\tau$ is a
faithful, semifinite, normal trace. Thus we require $\D$ affiliated
to $\cn$, and the compact resolvent condition is relative to $\K(\cn,\tau)$.
Examples of semifinite spectral triples arising from graph and
$k$-graph
$C^*$-algebras were described in \cite{PR,PRS}. These were constructed
using the natural action of the torus $T^k$ on a $k$-graph algebra, by
`pushing forward' the Dirac operator on the torus. More sophisticated examples 
coming from covering spaces
of manifolds of bounded geometry were considered in \cite{CGRS2} also.

For the $k$-graph algebras, $k\geq 2$, we may of course take a
Lorentzian
Dirac operator (or more generally pseudo-Riemannian Dirac operator) and
push
this forward instead. This gives rise to a `semifinite pseudo-Riemannian
spectral triple', but as the details would take us too far afield, we
leave this to the reader to explore.

\section{Constructing a spectral triple from a pseudo-Riemannian spectral triple}
\label{sRst-st}
Our main theorem associates to a pseudo-Riemannian spectral triple a 
bona fide spectral triple. As before, we write $\pert:=\frac i2(\D^2-\D^{*2})$ for brevity. 

\begin{thm}
\label{convert} 
Let $(\A,\HH,\D)$ be a pseudo-Riemannian spectral triple. Consider
$$
\D_E:=\frac{1}{2}(\D+\D^*)+\frac{i}{2}(\D-\D^*)=\frac{e^{i\pi/4}}{\sqrt{2}}(\D-i\D^*)
$$
as in Definition \ref{lem:symm}. 
Then $(\A,\HH,\D_E)$ is a spectral triple. 
\end{thm}

\begin{proof}
We already know from Corollary \ref{cor:D_E-sa} that $\D_E$ is essentially selfadjoint on ${\rm dom}\,\D\,\cap{\rm dom}\,\D^*$. 
From the definition of pseudo-Riemannian spectral triple, the
commutator
$$
[\D_E,a]=\frac{1+i}{2}[\D,a]+\frac{1-i}{2}[\D^*,a]
$$
is bounded for all $a\in\A$. Thus, to show that we have a spectral triple, we need to consider the resolvent condition. 
We first need to show, for $a\in\A$,  the compactness of
\begin{align*}
a(1+\la\D\ra^2+\pert)^{-1}
&=a(1+\la\D\ra^2)^{-1}-a(1+\la\D\ra^2)^{-1}\pert(1+\la\D\ra^2+\pert)^{-1}.
\end{align*}
The first term is compact by condition $4)$ in the definition of pseudo-Riemannian 
spectral triple. For the second term
we write
$$
a(1+\la\D\ra^2)^{-1}\pert=-a(1+\la\D\ra^2)^{-1}[\la\D\ra^2,\pert](1+\la\D\ra^2)^{-1}+a\pert(1+\la\D\ra^2)^{-1}.
$$
Both terms here are compact, the first by $2a)$ and $4)$, the second by $2b)$. 
Thus $a(1+\la\D\ra^2+\pert)^{-1}$ is compact.
Finally we employ the integral formula for fractional powers to complete the 
proof that we have a spectral triple. The same reasoning as
above gives the compactness of the integrand in
$$
a(1+\la\D\ra^2+\pert)^{-1/2}=\frac{1}{\pi}\int_0^\infty \lambda^{-1/2}a(1+\lambda+\la\D\ra^2+\pert)^{-1}d\lambda,
$$
and as the integral converges in norm, the left hand side is a compact operator.
Thus $(\A,\HH,\D_E)$ is a spectral triple. 
\end{proof}

\subsection{Smooth summability}
\label{sec:smooth_sum}

We now consider the smooth summability of the spectral triple $(\A,\HH,\D_E)$. 
%
%
While this can be checked directly for each example, as we do for the harmonic oscillator, we first
present a sufficient condition guaranteeing the smooth 
summability of the spectral triple $(\A,\HH,\D_E)$ 
given the smooth summability of $(\A,\HH,\D)$, with the same spectral dimension.
Our sufficient condition requires an 
additional assumption on the boundedness of 
$(1+\la\D\ra^2)(1+\D_E^2)^{-1}$. 
The harmonic oscillator (see \S\ref{sec:K-hom_harm_osc}) suggests that this condition is not necessary. 
We proceed by proving a few lemmas about the structure of pseudodifferential
operators associated to $(\A,\HH,\D)$. 

\begin{lemma}
\label{lem:smo-one}
Let $(\A,\HH,\D)$ be a 
pseudo-Riemannian spectral 
triple. 
Then 
$(1+\D_E^2)$ is an element of ${\rm OP}^2(\la\D\ra)$. 
Furthermore, if $(1+\la\D\ra^2)(1+\D_E^2)^{-1}$ is bounded, then $(1+\D_E^2)^{-1}$ is an element of ${\rm OP}^{-2}(\la\D\ra)$. 
\end{lemma}

\begin{proof} 
Since $\pert\in {\rm OP}^2(\la\D\ra)$ by hypothesis, 
$1+\D_E^2 = 1+\la\D\ra^2+\pert \in {\rm OP}^2(\la\D\ra)$. If furthermore $(1+\la\D\ra^2)(1+\D_E^2)^{-1}$ is bounded, it follows from Lemma \ref{lem:OP-invert} that $(1+\la\D\ra^2)(1+\D_E^2)^{-1}\in{\rm OP}^{0}(\la\D\ra)$, and hence that $(1+\D_E^2)^{-1}\in{\rm OP}^{-2}(\la\D\ra)$. 
\end{proof}




In the following discussion of smooth summability, the boundedness of 
$(1+\la\D\ra^2)(1+\D_E^2)^{-1}$ plays a crucial role. We pause to give a sufficient condition
for this boundedness to hold.

\begin{lemma}
\label{lem:use-A}
Let $(\A,\HH,\D)$ be a
pseudo-Riemannian spectral 
triple with 
$$
\Vert(1+\la\D\ra^2)^{-1/2}\pert(1+\la\D\ra^2)^{-1/2}\Vert<1.
$$ 
Then $(1+\la\D\ra^2)(1+\D_E^2)^{-1}$ lies in ${\rm OP}^0(\la\D\ra)$, and hence in particular is bounded.
\end{lemma}
\begin{proof}
We know that $(1+\la\D\ra^2)^{-1/2}\pert(1+\la\D\ra^2)^{-1/2}$ is bounded since $\pert\in{\rm OP}^2(\la\D\ra)$,
and assuming that $\Vert(1+\la\D\ra^2)^{-1/2}\pert(1+\la\D\ra^2)^{-1/2}\Vert<1$ ensures that $-1$
is not in the spectrum. Hence the operator
\begin{align*}
(1+(1+\la\D\ra^2)^{-1/2}\pert(1+\la\D\ra^2)^{-1/2})^{-1}
&=(1+\la\D\ra^2)^{1/2}(1+\D_E^2)^{-1}(1+\la\D\ra^2)^{1/2}
\end{align*}
is bounded. Observe that this also implies the boundedness of $(1+\D_E^2)^{-1/2}(1+\la\D\ra^2)^{1/2}$
and $(1+\la\D\ra^2)^{1/2}(1+\D_E^2)^{-1/2}$.

From Lemma \ref{lem:smo-one} we know that $(1+\la\D\ra^2)^{-1/2}(1+\D_E^2)(1+\la\D\ra^2)^{-1/2}$ lies in ${\rm OP}^0(\la\D\ra)$, and since this operator has bounded inverse, it follows from Lemma \ref{lem:OP-invert} that $(1+\la\D\ra^2)^{1/2}(1+\D_E^2)^{-1}(1+\la\D\ra^2)^{1/2}$ lies in ${\rm OP}^0(\la\D\ra)$. This also implies that $(1+\la\D\ra^2)(1+\D_E^2)^{-1}$ lies in ${\rm OP}^0(\la\D\ra)$.
\end{proof}

\begin{lemma}
\label{lem:ratios-1}
Let $(\A,\HH,\D)$ be a 
pseudo-Riemannian spectral 
triple 
such that $(1+\la\D\ra^2)(1+\D_E^2)^{-1}$ is bounded. 
Then the ratios
\begin{align*}
(1+\la\D\ra^2)^{-s}(1+\la\D\ra^2+\pert)^s,\ \ (1+\la\D\ra^2+\pert)^s(1+\la\D\ra^2)^{-s},\nonumber\\
(1+\la\D\ra^2+\pert)^{-s}(1+\la\D\ra^2)^s,\ \ (1+\la\D\ra^2)^s(1+\la\D\ra^2+\pert)^{-s}
\end{align*}
are bounded for all $0\leq s\in\R$.
\end{lemma}

\begin{proof}
Let $\sigma^r(T):=(1+\la\D\ra^2)^{r/2}\,T\,(1+\la\D\ra^2)^{-r/2}$ be the one parameter group
associated to $\la\D\ra$. Consider, for instance, the ratio $(1+\la\D\ra^2)^{-s}(1+\D_E^2)^s$. Then for $s\in\R$ we have
\begin{align*}
(1+\la\D\ra^2)^{-s}(1+\D_E^2)^s &= (1+\la\D\ra^2)^{-s+1}(1+\la\D\ra^2)^{-1}(1+\D_E^2)(1+\D_E^2)^{s-1} \\
&= \sigma^{-2s+2}((1+\la\D\ra^2)^{-1}(1+\D_E^2))\,(1+\la\D\ra^2)^{-s+1}(1+\D_E^2)^{s-1}\\
&=:K_s\,(1+\la\D\ra^2)^{-s+1}(1+\D_E^2)^{s-1},
\end{align*}
where $K_s\in {\rm OP}^0(\la\D\ra)$ because $(1+\la\D\ra^2)^{-1}(1+\D_E^2)\in {\rm OP}^0(\la\D\ra)$ by Lemma \ref{lem:smo-one}.
Repeating this process
shows that we can assume that $0<s<1$. Similar arguments hold for the other ratios. 
For $0<s<1$, the function $t\mapsto t^s$ is operator monotone, and elementary techniques then yield the result. 
\end{proof}

\begin{lemma}
\label{lem:ratios}
Let $(\A,\HH,\D)$ be a 
pseudo-Riemannian spectral 
triple 
such that $(1+\la\D\ra^2)(1+\D_E^2)^{-1}$ is bounded. 
Then for $s\geq0$ the ratio
\begin{align*}
(1+\la\D\ra^2)^{s}(1+\la\D\ra^2+\pert)^{-s}
\end{align*}
is an element of ${\rm OP}^0(\la\D\ra)$.
\end{lemma}

\begin{proof}
As in the proof of Lemma \ref{lem:ratios-1} we can reduce the problem to the case $0<s<1$. We already know from Lemma \ref{lem:ratios-1} that $(1+\la\D\ra^2)^{s}(1+\la\D\ra^2+\pert)^{-s}$ 
is bounded. 
%
%
Let us write $T:=(1+\D_E^2)$ for brevity. 
As in the proof of Lemma \ref{lem:OP-invert}, there are (combinatorial) constants $C_{l,n,k}$ such that
\begin{multline*}
\delta^n_{\la\D\ra}((\lambda+T)^{-1}) = \sum_{1\leq l\leq n} \sum_{1\leq k_1,\ldots,k_l\leq n,|k|=n} C_{l,n,k} \\
(\lambda+T)^{-1} \delta_{\la\D\ra}^{k_1}(T) (\lambda+T)^{-1} \delta_{\la\D\ra}^{k_2} (\lambda+T)^{-1} \cdots (\lambda+T)^{-1} \delta_{\la\D\ra}^{k_l}(T) (\lambda+T)^{-1} .
\end{multline*}
Using the integral formula for fractional powers, we can then write
\begin{multline*}
(1+\la\D\ra^2)^{s}\delta^n_{\la\D\ra}(T^{-s}) = \frac{1}{\pi} \int_0^\infty \lambda^{-s} \sum_{1\leq l\leq n} \sum_{1\leq k_1,\ldots,k_l\leq n,|k|=n} C_{l,n,k} \\
(1+\la\D\ra^2)^{s}(\lambda+T)^{-1} \delta_{\la\D\ra}^{k_1}(T) 
(\lambda+T)^{-1} \delta_{\la\D\ra}^{k_2} (\lambda+T)^{-1} 
\cdots (\lambda+T)^{-1} \delta_{\la\D\ra}^{k_l}(T) (\lambda+T)^{-1} d\lambda .
\end{multline*}
Since $\delta_{\la\D\ra}^{k}(T)\in {\rm OP}^2(\la\D\ra)$ so that
$\delta_{\la\D\ra}^{k}(T) (\lambda+T)^{-1} 
= \delta_{\la\D\ra}^{k}(T)(1+\la\D\ra^2)^{-1} (1+\la\D\ra^2)T^{-1} T(\lambda+T)^{-1}$ 
is uniformly bounded in $\lambda$ for each $k$, we see that the integral 
converges to a bounded operator for all $n$, so 
$(1+\la\D\ra^2)^{s}(1+\D_E^2)^{-s}$ is in ${\rm OP}^0(\la\D\ra)$.
%
%
\end{proof}

\begin{thm}
\label{convert-smooth} 
Let $(\A,\HH,\D)$ be a smoothly summable pseudo-Riemannian spectral triple 
such that $(1+\la\D\ra^2)(1+\D_E^2)^{-1}$ is bounded. 
Then the corresponding spectral triple $(\A,\HH,\D_E)$ is smoothly summable, with the same
spectral dimension.
\end{thm}

\begin{proof}
Since the operators $(1+\D_E^2)^{s/2}(1+\la\D\ra^2)^{-s/2}$ and $(1+\la\D\ra^2)^{s/2}(1+\D_E^2)^{-s/2}$ are bounded (by Lemma \ref{lem:ratios-1}), we see that $a(1+\la\D\ra^2)^{-s/2}\in\cl^1(\HH)$ implies $a(1+\D_E^2)^{-s/2}\in\cl^1(\HH)$ and conversely. Hence $(\A,\HH,\D_E)$ is finitely summable, with the same spectral dimension. 

Similarly, it is straightforward to show that 
$\B_2(\D_E,p)$ coincides with $\B_2(\la\D\ra,p)$, and so also that
$\B_1(\D_E,p)=\B_1(\la\D\ra,p)$. 
Now suppose that $T\in S^0$, which by assumption is contained in ${\rm OP}^0_0(\la\D\ra)$. Let us write $T^{(0)} := T$ and $T^{(n)} := [\D_E^2,T^{(n-1)}] = [\la\D\ra^2,T^{(n-1)}] + [\pert,T^{(n-1)}]$. It then follows that $T^{(n)}$ is a finite linear combination of elements of $S^n$, which by assumption is contained in ${\rm OP}^n_0(\la\D\ra)$. Because we can write
$$
R^k_{\D_E}(T) = T^{(k)} (1+\D_E^2)^{-k/2} = T^{(k)} (1+\la\D\ra^2)^{-k/2} (1+\la\D\ra^2)^{k/2} (1+\D_E^2)^{-k/2} ,
$$
and because $(1+\la\D\ra^2)^{k/2} (1+\D_E^2)^{-k/2}$ lies in ${\rm OP}^0(\la\D\ra)$ by Lemma \ref{lem:ratios}, we see that $R^k_{\D_E}(T)$ is an element of ${\rm OP}^0_0(\la\D\ra) = \B_1^\infty(\la\D\ra,p)$, so in particular $R^k_{\D_E}(T) \in \B_1(\la\D\ra,p) = \B_1(\D_E,p)$ for all $k$. Therefore $T$ is an element of $\B_1^\infty(\D_E,p)$, and hence $S^0 \subset \B_1^\infty(\D_E,p)$. Since $\A\cup[\D_E,\A]$ consists of linear combinations of elements of $S^0$, we conclude that $\A\cup[\D_E,\A]$ is also contained in $\B_1^\infty(\D_E,p)$, and hence that $(\A,\HH,\D_E)$ is smoothly summable. 
\end{proof}

%
%
%
%

Given a pseudo-Riemannian spectral triple $(\A,\HH,\D)$ we have seen
that there is
a more or less canonical construction of a spectral triple
$(\A,\HH,\D_E)$, and
this construction preserves the property of smooth summability needed for the local index formula.

This means that to each smoothly summable pseudo-Riemannian spectral
triple $(\A,\HH,\D)$ we can associate a $K$-homology class $[(\A,\HH,\D_E)]\in
K^*(A)$, where $A$ is the norm completion of $\A$, and compute the pairing of this class with
$K_*(A)$ using the (nonunital) local index formula, \cite{CGRS2}. 

\begin{rem}
Theorem \ref{convert} shows that the operator 
$\D_E=\frac{e^{i\pi/4}}{\sqrt{2}}(\D-i\D^*)$
can be used to obtain a spectral triple, where $\D_E^2=\la\D\ra^2+\pert$. 
However the proof is exactly the same for the
operator $\widetilde{\D_E}=\frac{e^{-i\pi/4}}{\sqrt{2}}(\D+i\D^*)$, with 
$\widetilde{\D_E}^2=\la\D\ra^2-\pert$. 
\end{rem}

This means that we potentially have two $K$-homology classes 
associated to a pseudo-Riemannian spectral
triple. In Section \ref{LIT} we will see a refinement of the definition, 
from \cite{PS}, which guarantees that these two classes are in fact negatives of each other.

\subsection{Examples}

We present a few examples showing what happens when a pseudo-Riemannian
spectral triple is converted into a spectral triple using Theorem
\ref{convert}. 

\subsubsection{Pseudo-Riemannian manifolds}

We have already shown in Section \ref{sec:classical} that for pseudo-Riemannian
manifolds which arise from complete Riemannian manifolds of bounded geometry, 
the triple $\big(C_c^\infty(M), L^2(M,S), \sD\big)$ is a pseudo-Riemannian spectral triple. 
Theorem \ref{convert} then returns the self-adjoint operator $\sD_E$. 
(Actually the signs we have used mean that we obtain $\sD_E=\Re\sD-\Im\sD$,
but by the remark above we still get a smoothly summable spectral triple. Using $\frac{e^{-i\pi/4}}{\sqrt{2}}(\sD+i\sD^*)$
would return the operator $\widetilde\sD_E=\Re\sD+\Im\sD$.) 

In these examples the operator $\pertsD$ is a first order 
differential operator, and so $\sD_E^2 = \la\sD\ra^2 - \pertsD$ is a second order 
positive elliptic differential operator. This implies that 
$(1+\la\sD\ra^2)(1+\sD_E^2)^{-1}$ is bounded, since 
$(1+\sD_E^2)^{-1}$ is an order $-2$ pseudodifferential operator.
Then Theorem \ref{convert-smooth} tells us that the 
associated spectral triple is smoothly summable.

\subsubsection{Finite geometries}
With the same notation as in \S\ref{subsub:finite}, we find the operator
$$
\D_E = \frac{e^{i\pi/4}}{\sqrt 2} \bma 0 &-iB^*\\ B&0\ema .
$$

\subsubsection{First order differential operators}

With the same notation as in \S\ref{subsub:1st-order}, we find that
$$
\D_E=\sum_j\tilde{M}_j\partial_j +\frac{1}{2}(K+K^*)+\frac{i}{2}(K-K^*)
$$
where $\tilde{M}_j=M_j$ if $M_j^*=-M_j$ and $\tilde{M}_j=iM_j$ if $M_j=M_j^*$.
Since $\pert$ is a first order differential operator in these examples, and $\D_E^2$ is
second order and uniformly elliptic, one can show that
$\pert(1+\D_E^2)^{-1}$ is bounded, and Theorem \ref{convert-smooth} gives us a smoothly summable spectral
triple.
\subsection{The $K$-homology class of the harmonic oscillator}
\label{sec:K-hom_harm_osc}

Here we consider the $K$-homology class of the spectral triple obtained
from the harmonic oscillator.  Theorem \ref{convert} gives
$$
\D_E=i\frac{d}{dx}+x, \quad \D_E^2=-\frac{d^2}{dx^2}+x^2+i(1+2x\frac{d}{dx}).
$$

We let $C_1^\infty(\R)$ be the smooth functions all of whose derivatives are integrable on $\R$. 
We have seen in Proposition \ref{prop:harm_osc_smooth} that the pseudo-Riemannian spectral triple 
$(C_1^\infty(\R),L^2(\R),\frac{d}{dx}\,+x)$ is smoothly summable with spectral dimension $p=1$.

In order to conclude from Theorem \ref{convert-smooth} that the 
spectral triple $(C_1^\infty(\R),L^2(\R),\D_E)$ is also smoothly 
summable, we would need to check that $(1+\la\D\ra^2)(1+\D_E^2)^{-1}$ is a 
bounded operator. We have been unable to prove this, and at present
have no reason to believe it is true.

On the other hand, we can simply check that 
$C_1^\infty(\R)\subset B_1^\infty(\D_E,1)$. However, this follows in the same way 
as in the proof of Proposition \ref{prop:harm_osc_smooth}. 
The integral kernels of $e^{-t\D_E^2}$ and $(1+\D_E^2)^{-s/2}$ are given by
\begin{align*}
k_t(x,y)&=\frac{1}{2\sqrt{\pi t}}e^{-(x-y)^2/4t + i(x^2-y^2)/2}, &
k_s(x,y)&=\frac{1}{\Gamma(s/2)}\int_0^\infty t^{s/2-1}e^{-t}k_t(x,y)dt.
\end{align*}
Then for any $g\in C_1^\infty(\R)$ one finds that
\begin{align*}
{\rm Trace}(g(1+\D_E^2)^{-s/2})
=\frac{\Gamma(s/2-1/2)}{2\sqrt{\pi}\Gamma(s/2)}\int_\R g(x)\,dx.
\end{align*}
Thus the spectral dimension is $\geq 1$, and taking $g(x)=e^{-x^2}$ shows that the spectral dimension
is precisely $1$. The smoothness is an easy check, using the same computations in the
proof of Proposition \ref{prop:harm_osc_smooth}. To show that the spectral triple is smoothly summable,
we observe that for all $f\in C_1^\infty(\R)$ we have 
$
[\D_E^2,f]=-f''+2if'\,\D_E.
$
Then a straightforward induction shows that $C_1^\infty(\R)\subset B_1^\infty(\D_E,1)$.
Hence the spectral triple $(C_1^\infty(\R),L^2(\R),\D_E)$ is also smoothly 
summable with spectral dimension $1$. We can then apply the local index formula of 
\cite[Theorem 4.33]{CGRS2}.
%
%
%
%

For a unitary $u$ in the unitization
of $C_1^\infty(\R)$, the local index formula computes the pairing of the class of the spectral
triple with the $K$-theory class of $u$ as
$$
\la [u],[(C^\infty_1(\R),L^2(\R),\D_E)]\ra=-\lim_{s\to1/2}(s-1/2){\rm Trace}(u^*[\D_E,u](1+\D_E^2)^{-s}).
$$
The odd $K$-theory of the real line is $K_1(C_0(\R))=\Z$.
For $m\in\Z$ we choose representatives of these classes to be  $u=e^{2im\tan^{-1}(x)}$, and this gives $u^*[\D_E,u]=\frac{-2m}{1+x^2}=:g_m(x)$. Using the trace calculations above we have
\begin{multline*}
\lim_{s\to1}\frac{s-1}2{\rm Trace}(u^*[\D_E,u](1+\D_E^2)^{-s/2})
=\lim_{s\to1}\frac{1}{2\sqrt{\pi}\Gamma(s/2)}
\int_\R g(x)\int_0^\infty \frac{s-1}2t^{s/2-1/2-1}e^{-t}dt\,dx\\
=\lim_{s\to1}\frac{1}{2\sqrt{\pi}\Gamma(s/2)}
\int_\R g(x)\int_0^\infty \frac{d}{dt}(t^{s/2-1/2})e^{-t}dt\,dx
=\frac{-m}{\pi}\int_\R \frac{1}{1+x^2}dx=-m.
\end{multline*}
Hence the spectral triple $(C_1^\infty(\R),L^2(\R),i\frac{d}{dx}+x)$ 
has a nontrivial $K$-homology class, and
it coincides with the class of $(C_1^\infty(\R),L^2(\R),i\frac{d}{dx})$, see \cite[page 298]{HR}. 
From the perspective of principal
symbols this is not surprising, however the unboundedness of the 
perturbation means that the result is not immediate.

\subsection{Application to compact manifolds}\label{LIT}

We give an index theoretic application in a special case closely
related to the  Lorentzian spectral triples of \cite{PS}.

\begin{defn}\label{beta-defn} A Lorentz-type
spectral triple   $(\A,\HH,\D,\beta)$ is given by a $*$-algebra  $\A$
represented on the Hilbert space $\HH$ as bounded operators,
along with a bounded operator $\beta$ on $\HH$ and a densely
defined closed operator $\D:{\rm dom}\,\D\subset\HH\to\HH$ such
that

0) ${\rm dom}\D^*\D\cap{\rm dom}\D\D^*$ is dense in $\HH$ and $\la\D\ra^2$ is 
essentially self-adjoint on this domain;

1) $\beta=-\beta^*$, $\beta^2=-1$, $\beta a=a\beta$ for all 
$a\in\A$ and $\beta\D$ is essentially self-adjoint on 
${\rm dom}\D^*\cap{\rm dom}\D$;

2a) $\beta[\D^2,\beta]:\HH_\infty\to\HH_\infty$ and 
$[\la\D\ra^2, \beta[\D^2,\beta]\,]\in{\rm OP}^2(\la\D\ra)$;

2b) $a\beta[\D^2,\beta](1+\la\D\ra^2)^{-1}\in \K(\HH)$ for all $a\in\A$;

3) 
$[\D,a]$ and $[\D^*,a]$ extend to bounded operators on $\HH$ for all $a\in\A$;

4) $a(1+\la\D\ra^2)^{-1/2}\in\K(\HH)$ for all $a\in\A$.

The triple is said to be even if there exists $\Gamma\in\B(\HH)$ such that
$\Gamma=\Gamma^*$, $\Gamma^2=1$, $\Gamma a=a\Gamma$ for all $a\in\A$,
$\beta\Gamma+\Gamma\beta=0$ and
$\Gamma\D+\D\Gamma=0$. Otherwise the triple is said to be odd.
\end{defn}


\begin{rem}
As noted in \cite[page 5]{PS}, this is not really capturing Lorentzian signature, but
rather that the number of timelike directions is odd. This can be refined using a real structure.
\end{rem}

The essential self-adjointness of $\beta\D$ implies that $\D^*=\beta\D\beta$ on
${\rm dom}\,\D^*\cap{\rm dom}\,\D$,
and that $\beta$ preserves this domain. In particular, $\D$ is Krein self-adjoint as in 
Section \ref{subsec:Krein} for the fundamental symmetry ${\mathcal J}=i\beta$.
The condition on $\beta[\D^2,\beta]$ can be understood
in terms of the definition of pseudo-Riemannian spectral triples via
$$ 
\beta[\D^2,\beta]=\beta\D^2\beta+\D^2=-\D^{*2}+\D^2.
$$
In particular a
Lorentz-type spectral triple  is a
pseudo-Riemannian spectral triple.

\begin{lemma}
Let $(\A,\HH,\D,\beta)$ be a Lorentz-type
spectral triple. Then the $K$-homology classes arising from the operators 
$$
\D_E=\frac{e^{i\pi/4}}{\sqrt{2}}(\D-i\D^*)\ \ \mbox{and}\ \ \widetilde{\D_E}=\frac{e^{-i\pi/4}}{\sqrt{2}}(\D+i\D^*)
$$
are negatives of one another.
\end{lemma}

\begin{proof}
One simply computes that $\beta \D_E\beta^*=-\widetilde{\D_E}$, 
which shows that $(\A,\HH,\D_E)$ is unitarily equivalent to $(\A,\HH,-\widetilde{\D_E})$.
\end{proof}

\begin{prop} Let $(\A,\HH,\D,\beta,\Gamma)$ be an even Lorentz-type
spectral
triple with  $\A$ unital. Assume that there
exists at least one even function $f$ with $f(0)\neq 0$ and $f(\D_E)$
trace class. If
$\D^2-\D^{*2}=\beta[\D^2,\beta]=0$, then
$$ 
{\rm Index}\left(\frac{1-\Gamma}{2}\D_E\frac{1+\Gamma}{2}\right)=0.
$$
That  is, the pairing of $[(\A,\HH,\D_E)]\in K^0(\A)$ with the class
$[1]\in K_0(\A)$ is zero.
\end{prop}

\begin{rem}
By analogy with the classical case we call $\beta$
Lorentz harmonic, since it commutes with the Laplacian of the indefinite metric. 
The hypothesis on the summability of $\D_E$ is implied by  
$\theta$-summability or finite summability of course.
\end{rem}

\begin{proof} 
Since $\D^2-\D^{*2}$ vanishes, the fundamental symmetry $\beta$ commutes with 
$\D_E^2 = \la\D\ra^2$, and so also with any function of $\D_E^2$.
The index can be computed using the McKean-Singer formula, and we
refer to \cite{CPRS3} for this version. For any even function $f$ not
vanishing at $0$ and such that $f(\D_E)$ is trace class we have
$$ 
{\rm Index}\left(\frac{1-\Gamma}{2}\D_E\frac{1+\Gamma}{2}\right)
=\frac{1}{f(0)}\mbox{Trace}(\Gamma f(\D_E)).
$$
Then using $\beta^2=-1$, 
$\beta\Gamma+\Gamma\beta=0$,
$\beta\D_E^2=\D_E^2\beta$, 
and cyclicity of the trace,
it is straightforward to show that $\mbox{Trace}(\Gamma f(\D_E))$ must vanish identically. 
Hence the index vanishes, and the proof is complete.
\end{proof}

We thank Christian B\"{a}r for the following example.
\begin{cor} 
There exists a compact spin manifold which admits a
time-oriented Lorentzian metric (and so has zero Euler characteristic) but
does not possess any Lorentz harmonic fundamental symmetries. In particular,
it has no nowhere vanishing Lorentz harmonic one-forms.
\end{cor}

\begin{proof} 
Suppose we have a compact Riemannian spin manifold $(M,g_E)$ with vanishing
Euler characteristic. Let $z$ be a normalised non-vanishing vector field, and $z^\#$
the dual one-form. Then we get a Lorentzian metric  
$
g(v,w):=g_E(v,w)-2z^\#(v)z^\#(w).
$
With the same notation as in Section \ref{sec:classical},  we set $\beta=\gamma_E(z)$, 
and so we obtain the Lorentz-type
spectral triple 
$(C^\infty(M),L^2(M,S),\D,\beta)$, using Proposition \ref{prop:Baum}. 

Now suppose  that we can choose some fundamental symmetry $\tilde\beta$ 
to be Lorentz harmonic. Then
$\frac{1}{4}\mbox{Index}((1-\Gamma)\D_E(1+\Gamma))=0$, and we have shown previously that $\D_E$ 
is the  Dirac operator of the 
Riemannian metric $\tilde{g}_E$ canonically associated to our Lorentzian metric and
fundamental symmetry $\tilde{\beta}$. Hence if we can choose 
$\tilde{\beta}$ Lorentz harmonic, we must
have $\hat{A}(M)=0$. Equivalently, if $\hat{A}(M)\neq 0$ we see that we
can not choose $\tilde{\beta}$ harmonic, and in particular $\beta=\gamma_E(z)$ is
not Lorentz harmonic.

Let $M_1=S^2\times F_g$ where $F_g$ is a compact orientable surface of
genus $g\geq 2$. Then $\chi(M_1)=2(2-g)$ and $\hat{A}(M_1)=0$. Let
$M_2$ be a $K3$ space, so $\chi(M_2)=24$ and $\hat{A}(M_2)=2$. Finally
let $M=M_1\# M_2$ be the connected sum. Then
$\chi(M)=\chi(M_1)+\chi(M_2)-2=26-2g$ and
$\hat{A}(M)=\hat{A}(M_1)+\hat{A}(M_2)=2$. So for $g=13$ we may
construct a Lorentzian metric on $M$, and since the A-roof  genus does
not vanish, there do not exist any Lorentz harmonic fundamental symmetries.
\end{proof}

\end{document}